\documentclass{tac}

\usepackage{xy}

\usepackage{amsfonts}
\usepackage{stmaryrd}         
\usepackage{hyperref}

\input diagxy

\author {Robert Par\'e}

\address{Department of Mathematics and Statistics, Dalhousie University\\
 Halifax, NS, Canada, B3H 4R2\\}

\title[Retrocells] {Retrocells}

\dedication{For Marta Bunge, constant friend for over half a century}

\copyrightyear{2020}



\eaddress{pare@mathstat.dal.ca}

 \newcommand{\slashdot}{\mathchoice
                                                              {\mathop{\slash\llap{$\scriptstyle{\bullet\,}$}}}
                                                              {\mathop{\slash\llap{$\scriptstyle{\bullet\,}$}}}
                                                              {\mathop{\scriptstyle{\slash\llap{$\scriptscriptstyle{\bullet\,}$}}}}
                                                              {\mathop{\scriptstyle{\slash\llap{$\scriptscriptstyle{\bullet\,}$}}}}
                                                              }
                                                                                                                  
\newcommand{\bsd}{\mathchoice
                                                              {\mathop{\backslash\llap{$\scriptstyle{\bullet}$}}}
                                                              {\mathop{\backslash\llap{$\scriptstyle{\bullet}$}}}
                                                              {\mathop{\scriptstyle{\backslash\llap{$\scriptscriptstyle{\bullet}$}}}}
                                                              {\mathop{\scriptstyle{\backslash\llap{$\scriptscriptstyle{\bullet}$}}}}
                                                              }
                                                                                                                  
\newcommand{\bdot}{\mathchoice
                                                   {\mathop{\scriptstyle{\bullet}}} 
                                                   {\mathop{\scriptstyle{\bullet}}}
                                                   {\mathop{\scriptscriptstyle{\bullet}}}
                                                   {\mathop{\scriptscriptstyle{\bullet}}}
                                                   }

\mathrmdef{id}
\mathrmdef{Id}

\def\CCat{{\mathbb C}{\rm at}}

\mathrmdef{Ob}

\def\tod{\xy\morphism(0,0)|m|<225,0>[`;\hspace{-1mm}\bullet\hspace{-1mm}]\endxy}
\def\tosl{\xy\morphism(0,0)|m|<225,0>[`;\hspace{-1mm}+\hspace{-1mm}]\endxy}

\newcommand{\ols}[1]{\mskip.5\thinmuskip\overline{\mskip-.5\thinmuskip {#1} \mskip-.5\thinmuskip}\mskip.5\thinmuskip} 
\newcommand{\ov}[1]{{\ols{#1}}}
\newcommand{\ovv}[1]{{\widetilde{#1}}}

\newcommand{\todd}[2]{\xymatrix@1{\ar[r]|\bb^{#1}_{#2}&}}

\newcommand{\todo}[1]{\xymatrix@1{\ar[r]|\bb^{#1}&}}

\newtheorem{definition}{Definition}
\newtheorem{theorem}{Theorem}[section]
\newtheorem{proposition}{Proposition}

\newtheorem{corollary}{Corollary}

\newtheorem{example}{Example}

\input xy
\xyoption{all}

\mathrmdef{Hom}
\mathbfdef{Set}

\def\CCat{{\mathbb C}{\rm at}}

\def\Cat{{\cal C}{\it at}}

\def\Veq{{\mbox{\rule{.23mm}{2.3mm}\hspace{.6mm}\rule{.23mm}{2.3mm}}}}
\def\veq{{\mbox{\rule{.2mm}{2mm}\hspace{.5mm}\rule{.2mm}{2mm}}}}

\newbox\bbox
\setbox\bbox = \hbox to 0pt{\hss $\scriptstyle\bullet$\hss}
\def\bb{\usebox{\bbox}}

\begin{document}


\maketitle

\begin{abstract}
The notion of retrocell in a double category with companions
is introduced and its basic properties established. Explicit descriptions
in some of the usual double categories are given. Monads in a
double category provide an important example where retrocells
arise naturally. Cofunctors appear as a special case. The motivating
example of vertically closed double categories is treated in
some detail.
\end{abstract}


\tableofcontents


\section*{Introduction}

In \cite{Par21} an in-depth study of the double category $ {\mathbb R}{\rm ing} $ of
rings, homomorphisms, bimodules and linear maps was made, and several
interesting features were uncovered. It became apparent that considering this
double category, rather than the category of rings and homomorphisms or the
bicategory of bimodules, could provide some important insights into the nature
of rings and modules.

An important property of the bicategory of bimodules is that it is biclosed,
i.e. the $ \otimes $ has right adjoints in each variable so that we have bijections
of linear maps
\begin{center}
\begin{tabular}{c} 
 $M \to N \obslash_T P $ \\[3pt]  \hline \\[-12pt]
 $N \otimes_S M \to P$  \\[3pt] \hline \\[-12pt]
$ N \to P \oslash_R M $ 
\end{tabular}
\end{center}
for bimodules
$$
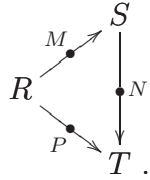
\bfig\scalefactor{.8}
\Ctriangle/@{<-}|{\bb}`@{>}|{\bb}`@{>}|{\bb}/<400,300>[S`R`T\rlap{\ .};M`N`P]

\efig
$$
We use (a slight modification of Lambek's notation for the hom bimodules
\cite{Lam66}): $ P \oslash_R M $ is the $ T\mbox{-}S $ bimodule of
$ R $-linear maps $ M \to P $, and $ T $-linear for $ N \obslash_T P $.
Both $ P \oslash_R M $ and $ N \obslash_T P $ are covariant in $ P $
but contravariant in the other variables. This is for $ 2 $-cells in the
bicategory $ {\cal{B}}{\it im} $ but it does not extend to cells in the double
category $ {\mathbb R}{\rm ing} $, which casts a shadow on our
contention that $ {\mathbb R}{\rm ing} $ works better than $ {\cal{B}}{\it im} $.

The way out of this dilemma is hinted at in the commuter cells of \cite{GraPar08}
(there called commutative cells) introduced to deal with the universal
property of internal comma objects. That is, to use companions to define
new cells, which we call retrocells below, and thus recover functoriality.

After a quick review of companions in Section 1, we introduce retrocells
in Section 2 and see that they are the cells of a new double category,
and if we apply this construction twice, we get the original double category, up to
isomorphism.

Section 3 extends the mates calculus to double categories where
we see retrocells appearing as the mates of standard cells. A careful
study of dualities in Section 4 completes this.

Retrocells in the standard double categories whose vertical arrows
are spans, relations, profunctors or $ {\bf V} $-matrices are
analyzed in Section 5. They correspond to various sorts of liftings
reminiscent of fibrations.

Section 6 studies retrocells in the context of monads in a double
category. It is seen that, while Kleisli objects are certain kinds of
universal cells, Eilenberg-Moore objects are universal retrocells.
In $ {\mathbb S}{\rm pan} {\bf A} $, monads are
category objects in $ {\bf A} $ and internal functors are cells
preserving identities and multiplication. Retrocells, on the other hand,
give cofunctors. 

In Section 7 we extend Shulman's closed equipments to general
double categories, and establish the functoriality of internal homs,
covariant in one variable and retrovariant in the other, formulated
in terms of ``twisted cospans''.

We end in Section 8 by re-examining commuter cells in the light
of retrocells and see that this leads to an interesting triple category,
though we do not pursue the triple category aspect.

The results of this paper were presented in preliminary form at CT2019
in Edinburgh and in the MIT Categories Seminar in October 2020. We
thank Bryce Clarke, Matt Di~Meglio and David Spivak for expressing their
interest in retrocells. We also thank the anonymous referee for a careful
reading and numerous suggestions resulting in a much better presentation.

\section{Companions}

The whole paper will be concerned with double categories that have
companions, so we recall the definition, principal properties we will use,
and establish some notation (see \cite{GraPar04} for more details).

\begin{definition} Let $ f \colon A \to B $ be a horizontal arrow in
a double category $ {\mathbb A} $. A {\em companion} for $ f $ is
a vertical arrow $ v \colon A \tod B $ together with two
{\em binding cells} $ \alpha $ and $ \beta $ as below, such that
$$
\bfig\scalefactor{.8}
\square/=`=`@{>}|{\bb}`>/[A`A`A`B;``v`f]

\place(250,250)[{\scriptstyle \alpha}]

\square(500,0)/>``=`=/[A`B`B`B;f```]

\place(750,250)[{\scriptstyle \beta}]

\place(1300,250)[=\ \  \id_f]

\place(1700,250)[\mbox{and}]

\square(2100,-250)/`@{>}|{\bb}`=`=/[A`B`B`B\rlap{\ .};`v``]

\place(2350,0)[{\scriptstyle \beta}]

\square(2100,250)/=`=`@{>}|{\bb}`>/[A`A`A`B;``v`f]

\place(2350,500)[{\scriptstyle \alpha}]

\place(3000,250)[= \ \ 1_v]

\efig
$$

\end{definition}

We can always assume the vertical identities are strict and usually
denote them by long equal signs in diagrams, as we just did. Of course
horizontal identities are always strict, and we use a similar 
diagrammatic notation.

The vertical identity on $ A $, $ \id_A $, is a companion to the
horizontal identity $ 1_A $, with both binding cells the common value
$ 1_{\id_A} = \id_{1_A} $,
$$
\bfig\scalefactor{.8}
\square/=`=`=`=/[A`A`A`A\rlap{\ .};```]

\place(250,250)[{\scriptstyle 1}]

\efig
$$
If $ f \colon A \to B $ and $ g \colon B \to C $ have respective companions
$ (v, \alpha, \beta) $ and $ (w, \gamma, \delta) $ then $ g f $ has
$ w \bdot v $ as companion with binding cells
$$
\bfig\scalefactor{.8}
\square/`=`=`>/[A`B`A`B;```f]

\place(250,250)[{\scriptstyle \id_f}]

\square(500,0)/=``@{>}|{\bb}`>/[B`B`B`C;``w`g]

\place(750,250)[{\scriptstyle \gamma}]

\square(0,500)/=`=`@{>}|{\bb}`>/[A`A`A`B;``v`f]

\place(250,750)[{\scriptstyle \alpha}]

\square(500,500)/=``@{>}|{\bb}`/[A`A`B`B;``v`]

\place(750,750)[{\scriptstyle 1_v}]

\place(1450,500)[\mbox{and}]

\square(1900,0)/=`@{>}|{\bb}``=/[B`B`C`C;`w``]

\place(2150,250)[{\scriptstyle 1_w}]

\square(2400,0)/>`@{>}|{\bb}`=`=/[B`C`C`C\rlap{\ .};g`w``]

\place(2650,250)[{\scriptstyle \delta}]

\square(1900,500)/>`@{>}|{\bb}`=`/[A`B`B`B;f`v``]

\place(2150,750)[{\scriptstyle \beta}]

\square(2400,500)/>``=`/[B`C`B`C;g```]

\place(2650,750)[{\scriptstyle \id_g}]

\efig
$$

Two companions $ (v, \alpha, \beta) $ and $ (v', \alpha', \beta') $ for
the same $ f $ are isomorphic by the globular isomorphism
$$
\bfig\scalefactor{.8}
\square/`@{>}|{\bb}`=`=/[A`B`B`B\rlap{\ .};f`v``]

\place(250,250)[{\scriptstyle \beta}]

\square(0,500)/=`=`@{>}|{\bb}`>/[A`A`A`B;``v'`]

\place(250,750)[{\scriptstyle \alpha'}]

\efig
$$

We usually choose a representative companion from each isomorphism
class and call it $ (f_*, \psi_f, \chi_f) $
$$
\bfig\scalefactor{.8}
\square/=`=`@{>}|{\bb}`>/[A`A`A`B;``f_*`f]

\place(250,250)[{\scriptstyle \psi_f}]

\square(1200,0)/>`@{>}|{\bb}`=`=/[A`B`B`B\rlap{\ .};f`f_*``]

\place(1450,250)[{\scriptstyle \chi_f}]

\efig
$$
The choice is arbitrary but it simplifies things if we choose the companion
of $ 1_A $ to be $ (\id_A, 1_{\id_A}, 1_{\id_A}) $. In all of our examples
there is a canonical choice and for that $ (1_A)_* = \id_A $.

To lighten the notation, we often write the binding cells $ \psi_f $ and
$ \chi_f $ as corner brackets in diagrams:
$$
\bfig\scalefactor{.8}
\square/=`=`@{>}|{\bb}`>/[A`A`A`B;``f_*`f]

\place(250,200)[\ulcorner]

\place(850,250)[\mbox{and}]

\square(1200,0)/>`@{>}|{\bb}`=`=/[A`B`B`B\rlap{\ .};f`f_*``]

\place(1450,250)[\lrcorner]

\efig
$$
We also use $ = $ and $ \Veq $ for horizontal and vertical identity cells.

There is a useful technique, called {\em sliding}, where we slide a
horizontal arrow around a corner into a vertical one. Specifically, there
are bijections natural in every way that makes sense,
$$
\bfig\scalefactor{.8}
\square/`@{>}|{\bb}`@{>}|{\bb}`>/<1000,500>[A`C`D`E;`v`w`h]

\morphism(0,500)|a|/>/<500,0>[A`B;f]

\morphism(500,500)|a|/>/<500,0>[B`C;g]

\place(500,250)[{\scriptstyle \alpha}]

\place(1500,250)[\longleftrightarrow]

\square(2000,-150)/>`@{>}|{\bb}``>/<500,800>[A`B`D`E;f`v``h]

\morphism(2500,650)|r|/@{>}|{\bb}/<0,-400>[B`C;g_*]

\morphism(2500,250)|r|/@{>}|{\bb}/<0,-400>[C`E;w]

\place(2250,250)[{\scriptstyle \beta}]
\efig
$$
and also
$$
\bfig\scalefactor{.8}
\square/>`@{>}|{\bb}`@{>}|{\bb}`/<1000,500>[A`B`C`E;f`v`w`]

\morphism(0,0)|b|/>/<500,0>[C`D;g]

\morphism(500,0)|b|/>/<500,0>[D`E;h]

\place(500,250)[{\scriptstyle \alpha}]

\place(1500,250)[\longleftrightarrow]

\square(2000,-150)/>``@{>}|{\bb}`>/<500,800>[A`B`D`E\rlap{\ .};f``w`h]

\morphism(2000,650)|l|/@{>}|{\bb}/<0,-400>[A`C;v]

\morphism(2000,250)|l|/@{>}|{\bb}/<0,-400>[C`D;g_*]

\place(2250,250)[{\scriptstyle \beta}]

\efig
$$
If we combine the two we get a bijection
$$
\bfig\scalefactor{.8}
\square(0,150)/>`@{>}|{\bb}`@{>}|{\bb}`>/[A`B`C`D;f`v`w`g]

\place(250,400)[{\scriptstyle \alpha}]

\place(1000,400)[\longleftrightarrow]

\square(1500,0)/`@{>}|{\bb}`@{>}|{\bb}`=/<500,400>[C`B`D`D;`g_*`w`]

\square(1500,400)/=`@{>}|{\bb}`@{>}|{\bb}`/<500,400>[A`A`C`B;`v`f_*`]

\place(1750,400)[{\scriptstyle \widehat{\alpha}}]

\efig
$$
which is, in a sense, the conceptual basis for retrocells. That, and the
idea that $ f $ and $ f_* $ are really the same morphism in different roles.

We refer the reader to \cite{Gra20} for all unexplained double category
matters.

\section{Retrocells}

Let $ {\mathbb A} $ be a double category in which every horizontal arrow
has a companion and choose a companion for each (with $ \id_A $ as the companion of $ 1_A $).

\begin{definition}

A {\em retrocell} $ \alpha $ in $ {\mathbb A} $, denoted
$$
\bfig\scalefactor{.8}
\square/>`@{>}|{\bb}`@{>}|{\bb}`>/[A`B`C`D;f`v`w`g]

\morphism(170,250)/<=/<200,0>[`;\alpha]

\efig
$$
is a (standard) double cell of $ {\mathbb A} $ of the form
$$
\bfig\scalefactor{.8}
\square/`@{>}|{\bb}`@{>}|{\bb}`=/<500,400>[B`C`D`D\rlap{\ .};`w`g_*`]

\square(0,400)/=`@{>}|{\bb}`@{>}|{\bb}`/<500,400>[A`A`B`C;`f_*`v`]

\place(250,400)[{\scriptstyle \alpha}]

\efig
$$

\end{definition}

\begin{theorem}

The objects, horizontal and vertical arrows of $ {\mathbb A} $ together
with retrocells, form a double category $ {\mathbb A}^{ret} $.

\end{theorem}

\begin{proof}

The horizontal composite $ \beta \alpha $ of retrocells
$$
\bfig\scalefactor{.8}
\square/>`@{>}|{\bb}`@{>}|{\bb}`>/[A`B`C`D;f`v`w`g]

\morphism(180,250)/<=/<200,0>[`;\alpha]

\square(500,0)/>``@{>}|{\bb}`>/[B`E`D`F;h``x`k]

\morphism(680,250)/<=/<200,0>[`;\beta]

\efig
$$
is given by
$$
\bfig\scalefactor{.8}
\square/=`@{>}|{\bb}`@{>}|{\bb}`=/[E`E`F`F;`x`x`]

\place(250,250)[=]

\square(0,500)/=`@{>}|{\bb}``/<500,1000>[A`A`E`E;`(hf)_*``]

\place(250,1000)[{\scriptstyle \cong}]

\morphism(500,1500)|r|/@{>}|{\bb}/<0,-500>[A`B;f_*]

\morphism(500,1000)|r|/@{>}|{\bb}/<0,-500>[B`E;h_*]

\square(500,0)/``@{>}|{\bb}`=/[E`D`F`F;``k_*`]

\place(750,500)[{\scriptstyle \beta}]

\square(500,500)/=``@{>}|{\bb}`/[B`B`E`D;``w`]

\square(500,1000)/=``@{>}|{\bb}`/[A`A`B`B;``f_*`]

\place(750,1250)[=]

\square(1000,0)/=``@{>}|{\bb}`=/[D`D`F`F;``k_*`]

\place(1250,250)[=]

\square(1000,500)/``@{>}|{\bb}`/[B`C`D`D;``g_*`]

\place(1250,1000)[{\scriptstyle \alpha}]

\square(1000,1000)/=``@{>}|{\bb}`/[A`A`B`C;``v`]

\square(1500,0)/=``@{>}|{\bb}`=/<500,1000>[C`C`F`F;``(kg)_*`]

\place(1750,500)[{\scriptstyle \cong}]

\square(1500,1000)/=``@{>}|{\bb}`/[A`A`C`C;``v`]

\place(1750,1250)[=]

\efig
$$
where the \ $ \cong $ \  represent the canonical isomorphisms
$ (hf)_* \cong h_* \bdot f_* $ and $ (kg)_* \cong k_* \bdot g_* $,
$$
\bfig\scalefactor{.8}
\square/`@{>}|{\bb}`=`=/<1000,500>[A`E`E`E;`(hf)_*``]

\place(500,250)[\lrcorner]

\morphism(0,500)|a|/>/<500,0>[A`B;f]

\morphism(500,500)|a|/>/<500,0>[B`E;h]

\square(0,500)/>`=`=`/[A`B`A`B;f```]

\place(250,750)[\veq]

\square(500,500)/=``@{>}|{\bb}`/[B`B`B`E;``h_*`]

\place(750,750)[\ulcorner]

\square(0,1000)/=`=`@{>}|{\bb}`/[A`A`A`B;``f_*`]

\place(250,1250)[\ulcorner]

\square(500,1000)/=``@{>}|{\bb}`/[A`A`B`B;``f_*`]

\place(750,1250)[=]

\place(1400,750)[\mbox{and}]

\square(1850,0)/=`@{>}|{\bb}`@{>}|{\bb}`=/[D`D`F`F;`k_*`k_*`]

\place(2100,250)[=]

\square(2350,0)/>``=`=/[D`F`F`F\rlap{\ .};k```]

\place(2600,250)[\lrcorner]

\square(1850,500)/>`@{>}|{\bb}`=`/[C`D`D`D;g`g_*``]

\place(2100,750)[\lrcorner]

\square(2350,500)/>``=`/[D`F`D`F;k```]

\place(2600,750)[\veq]

\square(1850,1000)/=`=`@{>}|{\bb}`/<1000,500>[C`C`C`F;``(kg)_*`]

\place(2350,1250)[\ulcorner]

\efig
$$

The vertical composite $ \alpha' \bdot \alpha $ of retrocells
$$
\bfig\scalefactor{.8}
\square/>`@{>}|{\bb}`@{>}|{\bb}`>/[C`D`C'`D';g`v'`w'`h]

\morphism(180,250)/<=/<200,0>[`;\alpha']

\square(0,500)/>`@{>}|{\bb}`@{>}|{\bb}`/[A`B`C`D;f`v`w`]

\morphism(180,750)/<=/<200,0>[`;\alpha]

\efig
$$
is
$$
\bfig\scalefactor{.8}
\square/=`@{>}|{\bb}`@{>}|{\bb}`=/[D`D`D'`D';`w'`w'`]

\place(250,250)[=]

\square(0,500)/`@{>}|{\bb}`@{>}|{\bb}`/[B`C`D`D;`w`g_*`]

\square(0,1000)/=`@{>}|{\bb}`@{>}|{\bb}`/[A`A`B`C;`f_*`v`]

\place(250,1000)[{\scriptstyle \alpha}]

\square(500,0)/`@{>}|{\bb}`@{>}|{\bb}`=/[D`C'`D'`D'\rlap{\ .};``h_*`]

\place(750,500)[{\scriptstyle \alpha'}]

\square(500,500)/=``@{>}|{\bb}`/[C`C`D`C';``v'`]

\square(500,1000)/=``@{>}|{\bb}`/[A`A`C`C;``v`]

\place(750,1250)[=]

\efig
$$

Horizontal and vertical identities are
$$
\bfig\scalefactor{.8}
\square(0,250)/=`@{>}|{\bb}`@{>}|{\bb}`=/[A`A`C`C;`v`v`]

\morphism(180,500)/<=/<200,0>[`;1_v]

\place(700,500)[=]

\square(900,0)/`@{>}|{\bb}`=`=/[A`C`C`C;`v``]

\place(1150,500)[=]

\square(900,500)/=`=`@{>}|{\bb}`/[A`A`A`C;``v`]

\place(1650,500)[\mbox{and}]

\square(1900,250)/>`=`=`>/[A`B`A`B;f```f]

\place(2150,500)[{\scriptstyle \id_f}]

\place(2600,500)[=]

\square(2900,0)/`=`@{>}|{\bb}`=/[B`A`B`B\rlap{\ .};``f_*`]

\place(3150,500)[=]

\square(2900,500)/=`@{>}|{\bb}`=`/[A`A`B`A;`f_*``]

\efig
$$

There are a number of things to check (horizontal and vertical unit laws and associativities as well as interchange),
all of which are straightforward
calculations and will be left to the reader. It is merely a question of writing
out the diagrams and following the steps indicated schematically below.

The identity laws are trivial because of our conventions that $ (1_A)_* = \id_A $
and vertical identities are as strict as in $ {\mathbb A} $.

For retrocells
$$
\bfig\scalefactor{.8}
\square/>`@{>}|{\bb}`@{>}|{\bb}`>/[A_0`A_1`C_0`C_1;f_1`v_0`v_1`g_1]

\morphism(180,250)/<=/<200,0>[`;\alpha_1]

\square(500,0)/>``@{>}|{\bb}`>/[A_1`A_2`C_1`C_2;f_2``v_2`g_2]

\morphism(680,250)/<=/<200,0>[`;\alpha_2]

\square(1000,0)/>``@{>}|{\bb}`>/[A_2`A_3`C_2`C_3\rlap{\ ,};f_3``v_3`g_3]

\morphism(1180,250)/<=/<200,0>[`;\alpha_3]

\efig
$$
$ \alpha_3 (\alpha_2 \alpha_1) $ is a composite of 17 cells arranged in
a $ 4 \times 7 $ array represented schematically as

 \begin{center}
 \setlength{\unitlength}{.9mm}
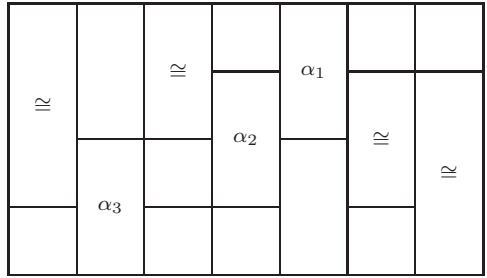
 \begin{picture}(70,40)
 \put(0,0){\framebox(70,40){}}
 
 \put(10,0){\line(0,1){40}}
 \put(20,0){\line(0,1){40}}
 \put(30,0){\line(0,1){40}}
 \put(40,0){\line(0,1){40}}
 \put(50,0){\line(0,1){40}}
 \put(60,0){\line(0,1){40}}

\put(0,10){\line(1,0){10}}
\put(10,20){\line(1,0){10}}
\put(20,10){\line(1,0){10}}
\put(20,20){\line(1,0){10}}
\put(30,10){\line(1,0){10}}
\put(30,30){\line(1,0){10}}
\put(40,20){\line(1,0){10}}
\put(50,10){\line(1,0){10}}
\put(50,30){\line(1,0){10}}
\put(60,30){\line(1,0){10}}
 
 \put(5,25){\makebox(0,0){$\scriptstyle\cong$}}
 \put(15,10){\makebox(0,0){$\scriptstyle \alpha_3$}}
 \put(25,30){\makebox(0,0){$\scriptstyle \cong$}}
 \put(35,20){\makebox(0,0){$\scriptstyle \alpha_2$}}
 \put(45,30){\makebox(0,0){$\scriptstyle \alpha_1$}}
 \put(55,20){\makebox(0,0){$\scriptstyle \cong$}}
 \put(65,15){\makebox(0,0){$\scriptstyle \cong$}}
 
 \put(115,20){(1)}
 
 \put(75,0){.}
 
 \end{picture}
 \end{center}

\noindent The empty rectangles are horizontal identities  and the $ \cong $
represent canonical isomorphisms generated by companions.

$ (\alpha_3 \alpha_2) \alpha_1 $ on the other hand is of the form

\begin{center}
\setlength{\unitlength}{.9mm}
\begin{picture}(70,40)
\put(0,0){\framebox(70,40){}}
 
\put(10,0){\line(0,1){40}}
\put(20,0){\line(0,1){40}}
\put(30,0){\line(0,1){40}}
\put(40,0){\line(0,1){40}}
\put(50,0){\line(0,1){40}}
\put(60,0){\line(0,1){40}}

\put(0,10){\line(1,0){10}}
\put(10,10){\line(1,0){10}}
\put(10,30){\line(1,0){10}}
\put(20,20){\line(1,0){10}}
\put(30,10){\line(1,0){10}}
\put(30,30){\line(1,0){10}}
\put(40,20){\line(1,0){10}}
\put(40,30){\line(1,0){10}}
\put(50,20){\line(1,0){10}}
\put(60,30){\line(1,0){10}}
 
\put(5,25){\makebox(0,0){$\scriptstyle\cong$}}
\put(15,20){\makebox(0,0){$\scriptstyle \cong$}}
\put(25,10){\makebox(0,0){$\scriptstyle \alpha_3$}}
\put(35,20){\makebox(0,0){$\scriptstyle \alpha_2$}}
\put(45,10){\makebox(0,0){$\scriptstyle \cong$}}
\put(55,30){\makebox(0,0){$\scriptstyle \alpha_1$}}
\put(65,20){\makebox(0,0){$\scriptstyle \cong$}}
 
\put(115,20){(2)}

\put(75,0){.}
 
 \end{picture}
 \end{center}
 
 \noindent It is now clear what to do. Switch $ \alpha_3 $ with $ \cong $
 in (1) and $ \alpha_1 $ with $ \cong $ in (2) to get
 
\begin{center}
\setlength{\unitlength}{.9mm}
\begin{picture}(30,40)
\put(0,0){\framebox(30,40){}}
 
\put(10,0){\line(0,1){40}}
\put(20,0){\line(0,1){40}}
\put(30,0){\line(0,1){40}}

\put(0,20){\line(1,0){10}}
\put(10,10){\line(1,0){10}}
\put(10,30){\line(1,0){10}}
\put(20,20){\line(1,0){10}}

\put(5,10){\makebox(0,0){$\scriptstyle \alpha_3$}}
\put(15,20){\makebox(0,0){$\scriptstyle \alpha_2$}}
\put(25,30){\makebox(0,0){$\scriptstyle \alpha_1$}}

 \end{picture}
 \end{center}

\noindent in the middle in both cases. The $ 4 \times 2 $ block on
the left in (1) becomes

\begin{center}
\setlength{\unitlength}{.9mm}
\begin{picture}(20,40)
\put(0,0){\framebox(20,40){}}
 
\put(10,0){\line(0,1){40}}
\put(20,0){\line(0,1){40}}

\put(0,10){\line(1,0){10}}
\put(10,20){\line(1,0){10}}

\put(5,25){\makebox(0,0){$\scriptstyle \cong$}}
\put(15,30){\makebox(0,0){$\scriptstyle \cong$}}
 
\put(25,0){.}
 
\end{picture}
\end{center}

\noindent which is not formally the same as $ 4 \times 2 $ block in
(2), but they are equal by one of the coherence
identities for $ (\ )_* $. We write it out
$$
\bfig\scalefactor{.8}
\square/=`@{>}|{\bb}``=/<500,1500>[A_0`A_0`A_3`A_3;`(f_3 f_2 f_1)_*``]

\place(250,750)[{\scriptstyle \cong}]

\morphism(500,1500)|r|/@{>}|{\bb}/<0,-1000>[A_0`A_2;(f_2 f_1)_*]

\morphism(500,500)|r|/@{>}|{\bb}/<0,-500>[A_2`A_3;f_{3*}]

\square(500,0)/=``@{>}|{\bb}`=/<500,500>[A_2`A_2`A_3`A_3;``f_{3*}`]

\place(750,250)[=]

\square(500,0)/=```/<500,1500>[A_0`A_0`A_3`A_3;```]

\morphism(1000,1500)|r|/@{>}|{\bb}/<0,-500>[A_0`A_1;f_{1*}]

\morphism(1000,1000)|r|/@{>}|{\bb}/<0,-500>[A_1`A_2;f_{2*}]

\place(750,1150)[{\scriptstyle \cong}]

\place(1400,850)[=]

\square(2000,0)/=`@{>}|{\bb}``=/<500,1500>[A_0`A_0`A_3`A_3;`(f_3 f_2 f_1)_*``]

\place(2250,750)[{\scriptstyle \cong}]

\morphism(2500,1500)|r|/@{>}|{\bb}/<0,-500>[A_0`A_1;f_{1*}]

\morphism(2500,1000)|r|/@{>}|{\bb}/<0,-1000>[A_1`A_3;(f_3 f_2)_*]

\square(2500,1000)/=``@{>}|{\bb}`=/[A_0`A_0`A_1`A_1;``f_{1*}`]

\place(2750,1250)[{\scriptstyle =}]

\morphism(3000,1000)|r|/@{>}|{\bb}/<0,-500>[A_1`A_2;f_{2*}]

\morphism(3000,500)|r|/@{>}|{\bb}/<0,-500>[A_2`A_3\rlap{\ .};f_{3*}]

\morphism(2500,0)/=/<500,0>[A_3`A_3;]

\place(2750,700)[{\scriptstyle \cong}]

\efig
$$
There may be something to worry about here because
$ (f_2 f_1)_* \cong f_{2*} \bdot f_{1*} $ involves $ \chi_{f_2 f_1} $ whereas
$ (f_3 f_2)_* \cong f_{3*} \bdot f_{2*} $ involves $ \chi_{f_3 f_2} $ which are
unrelated. However both $ \chi_{f_2 f_1} $ and $ \chi_{f_3 f_2} $ cancel in
the composites. The left hand side is

\begin{center}
\setlength{\unitlength}{1mm}
\begin{picture}(40,50)
\put(0,0){\framebox(40,50){}}
 
\put(10,10){\line(0,1){10}}
\put(10,20){\line(0,1){10}}
\put(20,0){\line(0,1){50}}
\put(30,30){\line(0,1){20}}
\put(40,0){\line(0,1){50}}

\put(0,10){\line(1,0){20}}
\put(0,20){\line(1,0){40}}
\put(0,30){\line(1,0){40}}

\put(20,40){\line(1,0){20}}

\put(10,5){\makebox(0,0){$\scriptstyle \chi_{f_3 f_2 f_1}$}}
\put(15,15){\makebox(0,0){$\scriptstyle \psi_{f_3}$}}
\put(5,25){\makebox(0,0){$\scriptstyle \psi_{f_2 f_1}$}}
\put(30,25){\makebox(0,0){$\scriptstyle\chi_{f_2 f_1}$}}
\put(35,35){\makebox(0,0){$\scriptstyle \psi_{f_2}$}}
\put(25,45){\makebox(0,0){$\scriptstyle \psi_{f_1}$}}

 \end{picture}
 \end{center}
 
 \noindent and when we cancel $ \chi_{f_2 f_1} $ with $ \psi_{f_2 f_1} $
 leaving $ \id_{f_2 f_1} $, that composite reduces to
 
 \begin{center}
\setlength{\unitlength}{1mm}
\begin{picture}(30,40)
\put(0,0){\framebox(30,40){}}
 
\put(10,10){\line(0,1){30}}
\put(20,10){\line(0,1){30}}
\put(30,0){\line(0,1){40}}

\put(0,10){\line(1,0){30}}
\put(0,20){\line(1,0){30}}
\put(0,30){\line(1,0){30}}

\put(15,5){\makebox(0,0){$\scriptstyle \chi_{f_3 f_2 f_1}$}}
\put(25,15){\makebox(0,0){$\scriptstyle \psi_{f_3}$}}
\put(15,25){\makebox(0,0){$\scriptstyle \psi_{f_2}$}}
\put(5,35){\makebox(0,0){$\scriptstyle \psi_{f_1}$}}

 \end{picture}
 \end{center}
 
\noindent as does the right hand side.

The $ 4 \times 2 $ block on the right is the same with the roles
of $ \psi $ and $ \chi $ interchanged. This completes the proof
of associativity of horizontal composition of retrocells.

The associativity for vertical composition is much simpler as it does not
involve $ \psi $'s or $ \chi $'s, only the associativity isomorphisms of
$ {\mathbb A} $. In particular if $ {\mathbb A} $ were strict, then $ {\mathbb A}^{ret} $
would be too, and the proof of associativity would be merely a question
of writing down the two composites and observing that they are exactly
the same.

For interchange consider retrocells
$$
\bfig\scalefactor{.8}
\square/>`@{>}|{\bb}`@{>}|{\bb}`>/[B_0`B_1`C_0`C_1;g_1`w_0`w_1`h_1]

\morphism(180,250)/<=/<200,0>[`;\beta_1]

\square(500,0)/>``@{>}|{\bb}`>/[B_1`B_2`C_1`C_2\rlap{\ .};g_2``w_2`h_1]

\morphism(680,250)/<=/<200,0>[`;\beta_2]

\square(0,500)/>`@{>}|{\bb}`@{>}|{\bb}`/[A_0`A_1`B_0`B_1;f_1`v_0`v_1`]

\morphism(180,750)/<=/<200,0>[`;\alpha_1]

\square(500,500)/>``@{>}|{\bb}`/[A_1`A_2`B_1`B_2;f_2``v_2`]

\morphism(680,750)/<=/<200,0>[`;\alpha_2]

\efig
$$
Then the pattern for $ (\beta_2 \beta_1) \bdot (\alpha_2 \alpha_1) $ is

 \begin{center}
 \setlength{\unitlength}{.9mm}
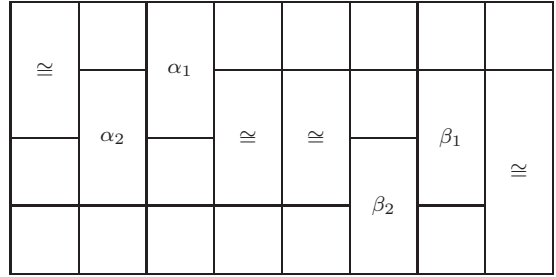
 \begin{picture}(80,40)
 \put(0,0){\framebox(80,40){}}
 
 \put(10,0){\line(0,1){40}}
 \put(20,0){\line(0,1){40}}
 \put(30,0){\line(0,1){40}}
 \put(40,0){\line(0,1){40}}
 \put(50,0){\line(0,1){40}}
 \put(60,0){\line(0,1){40}}
 \put(70,0){\line(0,1){40}}

\put(0,10){\line(1,0){50}}
\put(60,10){\line(1,0){10}}
\put(0,20){\line(1,0){10}}
\put(20,20){\line(1,0){10}}
\put(50,20){\line(1,0){10}}
\put(10,30){\line(1,0){10}}
\put(30,30){\line(1,0){50}}

\put(5,30){\makebox(0,0){$\scriptstyle \cong$}}
\put(15,20){\makebox(0,0){$\scriptstyle \alpha_2$}}
\put(25,30){\makebox(0,0){$\scriptstyle \alpha_1$}}
\put(35,20){\makebox(0,0){$\scriptstyle \cong$}}
\put(45,20){\makebox(0,0){$\scriptstyle \cong$}}
\put(55,10){\makebox(0,0){$\scriptstyle \beta_2$}}
\put(65,20){\makebox(0,0){$\scriptstyle \beta_1$}}
\put(75,15){\makebox(0,0){$\scriptstyle \cong$}}
%
 
 \end{picture}
 \end{center}

\noindent and for $ (\beta_2 \bdot \alpha_2) (\beta_1 \bdot \alpha_1) $ it is

 \begin{center}
 \setlength{\unitlength}{.9mm}
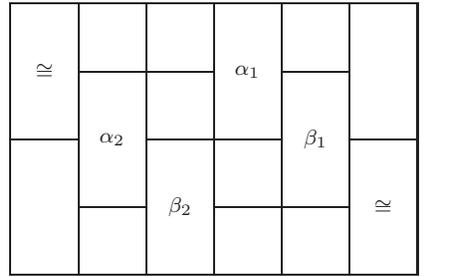
 \begin{picture}(60,40)
 \put(0,0){\framebox(60,40){}}
 
 \put(10,0){\line(0,1){40}}
 \put(20,0){\line(0,1){40}}
 \put(30,0){\line(0,1){40}}
 \put(40,0){\line(0,1){40}}
 \put(50,0){\line(0,1){40}}

\put(10,10){\line(1,0){10}}
\put(30,10){\line(1,0){20}}
\put(0,20){\line(1,0){10}}
\put(20,20){\line(1,0){20}}
\put(50,20){\line(1,0){10}}
\put(10,30){\line(1,0){20}}
\put(40,30){\line(1,0){10}}

\put(5,30){\makebox(0,0){$\scriptstyle \cong$}}
\put(15,20){\makebox(0,0){$\scriptstyle \alpha_2$}}
\put(25,10){\makebox(0,0){$\scriptstyle \beta_2$}}
\put(35,30){\makebox(0,0){$\scriptstyle \alpha_1$}}
\put(45,20){\makebox(0,0){$\scriptstyle \beta_1$}}
\put(55,10){\makebox(0,0){$\scriptstyle \cong$}}

\put(65,0){.}
 
 \end{picture}
 \end{center}

\noindent The two $ \cong $ in the middle of the first one are
inverse to each other,
$$
g_{2*} \bdot g_{1*} \to^{\cong} (g_2 g_1)_* \to^{\cong} g_{2*} \bdot g_{1*}\ ,
$$
so each of $ (\beta_2 \beta_1) \bdot (\alpha_2 \alpha_1) $ and
$ (\beta_2 \bdot \alpha_2) (\beta_1 \bdot \alpha_1) $ is equal to

 \begin{center}
 \setlength{\unitlength}{.9mm}
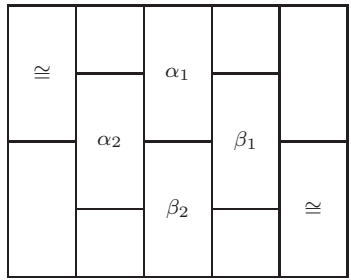
 \begin{picture}(50,40)
 \put(0,0){\framebox(50,40){}}
 
 \put(10,0){\line(0,1){40}}
 \put(20,0){\line(0,1){40}}
 \put(30,0){\line(0,1){40}}
 \put(40,0){\line(0,1){40}}
 \put(50,0){\line(0,1){40}}

\put(10,10){\line(1,0){10}}
\put(30,10){\line(1,0){10}}
\put(0,20){\line(1,0){10}}
\put(20,20){\line(1,0){10}}
\put(40,20){\line(1,0){10}}
\put(10,30){\line(1,0){10}}
\put(30,30){\line(1,0){10}}

\put(5,30){\makebox(0,0){$\scriptstyle \cong$}}
\put(15,20){\makebox(0,0){$\scriptstyle \alpha_2$}}
\put(25,30){\makebox(0,0){$\scriptstyle \alpha_1$}}
\put(25,10){\makebox(0,0){$\scriptstyle \beta_2$}}
\put(35,20){\makebox(0,0){$\scriptstyle \beta_1$}}
\put(45,10){\makebox(0,0){$\scriptstyle \cong$}}

 \end{picture}
 \end{center}

\noindent completing the proof.

\end{proof}

\begin{theorem}
\label{Thm-DoubDual}
(1) $ {\mathbb A}^{ret} $ has a canonical choice of companions.

\noindent (2) There is a canonical isomorphism of double categories with
companions
$$
{\mathbb A} \to^\cong {\mathbb A}^{ret\ ret}
$$
which is the identity on objects and horizontal and vertical arrows.

\end{theorem}

\begin{proof}
The companion of $ f \colon A \to B $ in $ {\mathbb A}^{ret} $ is $ f_* $,
$ f $'s companion in $ {\mathbb A} $ with binding retrocells
$$
\bfig\scalefactor{.8}
\square(0,250)/>`@{>}|{\bb}`=`=/[A`B`B`B;f`f_*``]

\morphism(170,500)/<=/<200,0>[`;]

\place(800,500)[=]

\square(1200,0)/`@{>}|{\bb}`@{>}|{\bb}`=/[B`B`B`B;`\id_B`\id_B`]

\square(1200,500)/=`@{>}|{\bb}`@{>}|{\bb}`/[A`A`B`B;`f_*`f_*`]

\place(1450,500)[{\scriptstyle 1}]

\efig
$$
and
$$
\bfig\scalefactor{.8}
\square(0,250)/=`=`@{>}|{\bb}`>/[A`A`A`B;``f_*`f]

\morphism(170,500)/<=/<200,0>[`;]

\place(800,500)[=]

\square(1200,0)/`@{>}|{\bb}`@{>}|{\bb}`=/[A`A`B`B\rlap{\ .};`f_*`f_*`]

\square(1200,500)/=`@{>}|{\bb}`@{>}|{\bb}`/[A`A`A`A;`\id_A`\id_A`]

\place(1450,500)[{\scriptstyle 1}]

\efig
$$
The binding equations only involve canonical isos so hold by coherence.

A cell $ \alpha $ in $ {\mathbb A}^{ret\,ret} $, i.e. a retrocell in $ {\mathbb A}^{ret} $
is
$$
\bfig\scalefactor{.8}
\square(0,250)/>`@{>}|{\bb}`@{>}|{\bb}`>/[A`B`C`D;f`v`w`g]

\morphism(170,500)/<=/<200,0>[`;\alpha]

\place(800,500)[=]

\square(1200,0)/`@{>}|{\bb}`@{>}|{\bb}`=/[B`C`D`D;`w`g_*`]

\square(1200,500)/=`@{>}|{\bb}`@{>}|{\bb}`/[A`A`B`C;`f_*`v`]

\place(1450,500)[{\scriptstyle \alpha}]

\place(2200,500)[\mbox{in\ \ ${\mathbb A}^{ret}$}]

\efig
$$
$$
\bfig\scalefactor{.8}
\place(800,750)[=]

\place(100,0)[\ ]

\square(1250,0)/`@{>}|{\bb}`@{>}|{\bb}`=/[C`D`D`D;`g_*`\id_D`]

\square(1250,500)/`@{>}|{\bb}`@{>}|{\bb}`/[A`B`C`D;`v`w`]

\square(1250,1000)/=`@{>}|{\bb}`@{>}|{\bb}`/[A`A`A`B;`\id_A`f_*`]

\place(1500,750)[{\scriptstyle \alpha}]

\place(2200,750)[\mbox{in\ \ ${\mathbb A}$}]

\efig
$$
and these are in canonical bijection with
$$
\bfig\scalefactor{.8}
\square/>`@{>}|{\bb}`@{>}|{\bb}`>/[A`B`C`D\rlap{\ .};f`v`w`g]

\place(250,250)[{\scriptstyle \alpha'}]

\efig
$$

Checking that composition and identities are preserved is a
straightforward calculation and is omitted.

\end{proof}

\begin{example}\rm

If $ \cal{A} $ is a $ 2 $-category, the double category of quintets
$ {\mathbb Q}\cal{A} $ has the same objects as $ \cal{A} $, the 
$ 1 $-cells of $ \cal{A} $ as both horizontal and vertical arrows, and
cells $ \alpha $
$$
\bfig\scalefactor{.8}
\square/>`@{>}|{\bb}`@{>}|{\bb}`>/[A`B`C`D;f`h`k`g]

\place(250,250)[{\scriptstyle \alpha}]

\place(850,250)[=]

\square(1200,0)[A`B`C`D;f`h`k`g]

\morphism(1550,350)/=>/<-140,-140>[`;\alpha]

\efig
$$
i.e. a $ 2 $-cell $ \alpha \colon k f \to g h $. Horizontal and vertical
composition are given by pasting. Every horizontal arrow $ f \colon A \to B $
has a companion, $ f_* $, namely $ f $ itself considered as a vertical
arrow. A retrocell
$$
\bfig\scalefactor{.8}
\square/>`@{>}|{\bb}`@{>}|{\bb}`>/[A`B`C`D;f`h`k`g]

\morphism(180,250)/<=/<200,0>[`;\alpha]

\efig
$$
is
$$
\bfig\scalefactor{.8}
\square/`@{>}|{\bb}`@{>}|{\bb}`=/[B`C`D`D;`k`g_*`]

\square(0,500)/=`@{>}|{\bb}`@{>}|{\bb}`/[A`A`B`C;`f_*`h`]

\place(250,500)[{\scriptstyle \alpha}]

\place(900,500)[=]

\square(1400,250)[A`A`D`D;1_A`k f`g h`1_D]

\morphism(1750,600)/=>/<-140,-140>[`;\alpha]

\efig
$$
i.e. a coquintet
$$
\bfig\scalefactor{.8}
\square[A`B`C`D\rlap{\ .};f`h`k`g]

\morphism(220,200)/=>/<140,140>[`;\alpha]

\efig
$$
Thus
$$
({\mathbb Q} {\cal{A}})^{ret} = {\rm co}{\mathbb Q}{\cal{A}} = {\mathbb Q}({\cal{A}}^{co})\ .
$$

\end{example}

\section{Adjoints, companions, mates}

The well-known mates calculus says that if we have functors
$ F, G, H, K, U, V $ as below with $ F \dashv U $ and $ G \dashv V $,
then there is a bijection between natural transformations $ t $
and $ u $ as below
$$
\bfig\scalefactor{.8}
\square[{\bf A}`{\bf B}`{\bf C}`{\bf D};H`U`V`K]

\morphism(180,250)/=>/<200,0>[`;t]

\place(900,250)[\longleftrightarrow]

\square(1300,0)[{\bf C}`{\bf D}`{\bf A}`{\bf B}\rlap{\ .};K`F`G`H]

\morphism(1680,250)/=>/<-200,0>[`;u]

\efig
$$
This is usually stated for bicategories but with the
help of retrocells we can extend it to double categories (with
companions).

To say that two horizontal arrows are adjoint in a double category
$ {\mathbb A} $ means they are so in the $ 2 $-category of
horizontal arrows $ {\cal{H}}{\it or} {\mathbb A} $. So $ h $ left
adjoint to $ f $ means we are given cells
$$
\bfig\scalefactor{.8}
\square/`=`=`=/<800,400>[A`A`A`A;```]

\morphism(0,400)/>/<400,0>[A`B;f]

\morphism(400,400)/>/<400,0>[B`A;h]

\place(400,200)[{\scriptstyle \epsilon}]

\place(1100,200)[\mbox{and}]

\square(1400,0)/=`=`=`/<800,400>[B`B`B`B;```]

\morphism(1400,0)|b|/>/<400,0>[B`A;h]

\morphism(1800,0)|b|/>/<400,0>[A`B;f]

\place(1800,200)[{\scriptstyle \eta}]

\efig
$$
satisfying the ``triangle'' identities
$$
\bfig\scalefactor{.8}
\square/`=`=`=/<1000,500>[A`A`A`A;```]

\morphism(0,500)/>/<500,0>[A`B;f]

\morphism(500,500)/>/<500,0>[B`A;h]

\place(500,250)[{\scriptstyle \epsilon}]

\square(1000,0)/>``=`>/[A`B`A`B;f```f]

\place(1250,250)[\veq]

\square(0,500)/>`=`=`/[A`B`A`B;f```]

\place(250,750)[\veq]

\square(500,500)/=``=`/<1000,500>[B`B`B`B;```]

\place(1000,750)[{\scriptstyle \eta}]

\place(1900,500)[=]

\square(2300,250)/>`=`=`>/[A`B`A`B;f```f]

\place(2550,500)[\veq]

\efig
$$
and
$$
\bfig\scalefactor{.8}
\square/>`=`=`>/[B`A`B`A;h```h]

\place(250,250)[\veq]

\square(500,0)/``=`=/<1000,500>[A`A`A`A;```]

\place(1000,250)[{\scriptstyle \epsilon}]

\morphism(500,500)/>/<500,0>[A`B;f]

\morphism(1000,500)/>/<500,0>[B`A;h]

\square(0,500)/=`=`=`/<1000,500>[B`B`B`B;```]

\place(500,750)[{\scriptstyle \eta}]

\square(1000,500)/>``=`/[B`A`B`A;h```]

\place(1250,750)[\veq]

\place(1900,500)[=]

\square(2300,250)/>`=`=`>/[B`A`B`A;h```h]

\place(2550,500)[\veq]

\place(2950,0)[.]

\efig
$$

To say that the vertical arrows are adjoint means that they are so
in the vertical bicategory $ {\cal{V}}{\it ert} {\mathbb A} $. So
$ x $ is left adjoint to $ v $ if we are given cells
$$
\bfig\scalefactor{.8}
\square/=``=`=/<550,1000>[A`A`A`A;```]

\place(275,500)[{\scriptstyle \epsilon}]

\morphism(0,1000)/@{>}|{\bb}/<0,-500>[A`C;v]

\morphism(0,500)/@{>}|{\bb}/<0,-500>[C`A;x]

\place(1000,500)[\mbox{and}]

\square(1450,0)/=`=``=/<550,1000>[C`C`C`C;```]

\morphism(2000,1000)|r|/@{>}|{\bb}/<0,-500>[C`A;x]

\morphism(2000,500)|r|/@{>}|{\bb}/<0,-500>[A`C;v]

\place(1750,500)[{\scriptstyle \eta}]

\efig
$$
also satisfying the triangle identities.

Suppose we are given horizontal arrows $ f $ and $ h $ with
cells $ \alpha_1 $ and $ \beta_1 $ as below. In the presence of
companions we can use sliding to transform them. We have
bijections
$$
\bfig\scalefactor{.8}
\square(0,0)/`=`=`=/<1000,500>[A`A`A`A;```]

\place(500,250)[{\scriptstyle \alpha_1}]

\morphism(0,500)/>/<500,0>[A`B;f]

\morphism(500,500)/>/<500,0>[B`A;h]

\place(1300,250)[\longleftrightarrow]

\square(1600,0)/>`=`@{>}|{\bb}`=/[A`B`A`A;f``h_*`]

\place(1850,250)[{\scriptstyle \alpha_2}]

\place(2450,250)[\longleftrightarrow]

\square(2800,-250)/=`=``=/<500,1000>[A`A`A`A;```]

\place(3050,250)[{\scriptstyle \alpha_3}]

\morphism(3300,750)|r|/@{>}|{\bb}/<0,-500>[A`B;f_*]

\morphism(3300,250)|r|/@{>}|{\bb}/<0,-500>[B`A;h_*]

\efig
$$

$$
\bfig\scalefactor{.8}
\square(0,0)/=`=`=`/<1000,500>[B`B`B`B;```]

\place(500,250)[{\scriptstyle \beta_1}]

\morphism(0,0)|b|/>/<500,0>[B`A;h]

\morphism(500,0)|b|<500,0>[A`B;f]

\place(1300,250)[\longleftrightarrow]

\square(1600,0)/=`@{>}|{\bb}`=`>/[B`B`A`B;`h_*``f]

\place(1850,250)[{\scriptstyle \beta_2}]

\place(2400,250)[\longleftrightarrow]

\square(2800,-250)/=``=`=/<500,1000>[B`B`B`B\rlap{\ .};```]

\place(3050,250)[{\scriptstyle \beta_3}]

\morphism(2800,750)/@{>}|{\bb}/<0,-500>[B`A;h_*]

\morphism(2800,250)/@{>}|{\bb}/<0,-500>[A`B;f_*]

\efig
$$

\begin{proposition}
\label{Prop-AdjComp}

$ h $ is left adjoint to $ f $ with adjunctions $ \alpha_1 $
and $ \beta_1 $ if and only if $ f_* $ is left adjoint to $ h_* $
with adjunctions $ \beta_3 $ and $ \alpha_3 $.

\end{proposition}

\begin{theorem}
\label{Thm-Mates}

Consider horizontal morphisms $ f $ and $ g $ and vertical
morphisms $ v $ and $ w $ as in
$$
\bfig\scalefactor{.8}
\square/>`@{>}|{\bb}`@{>}|{\bb}`>/[A`B`C`D\rlap{\ .};f`v`w`g]

\efig
$$
(1) If $ x $ is left adjoint to $ v $ and $ y $ left adjoint to $ w $,
then there is a bijection between cells $ \alpha $ and retrocells
$ \beta $ as in
$$
\bfig\scalefactor{.8}
\square/>`@{>}|{\bb}`@{>}|{\bb}`>/[A`B`C`D;f`v`w`g]

\place(250,250)[{\scriptstyle \alpha}]

\place(900,250)[\longleftrightarrow]

\square(1300,0)/>`@{>}|{\bb}`@{>}|{\bb}`>/[C`D`A`B\rlap{\ .};g`x`y`f]

\morphism(1620,250)/=>/<-150,0>[`;\beta]

\efig
$$
(2) If $ h $ is left adjoint to $ f $ and $ k $ left adjoint to $ g $,
then there is a bijection between cells $ \alpha $ and retrocells
$ \gamma $ as in
$$
\bfig\scalefactor{.8}
\square/>`@{>}|{\bb}`@{>}|{\bb}`>/[A`B`C`D;f`v`w`g]

\place(250,250)[{\scriptstyle \alpha}]

\place(900,250)[\longleftrightarrow]

\square(1300,0)/>`@{>}|{\bb}`@{>}|{\bb}`>/[B`A`D`C\rlap{\ .};h`w`v`k]

\morphism(1620,250)/=>/<-150,0>[`;\gamma]

\efig
$$

\end{theorem}

\begin{proof}
(1) Standard cells $ \alpha $ are in bijection with $ 2 $-cells
$ \widehat{\alpha} $ in the bicategory $ {\cal{V}}{\it ert} {\mathbb A} $
$$
\bfig\scalefactor{.8}
\square(0,250)/>`@{>}|{\bb}`@{>}|{\bb}`>/[A`B`C`D;f`v`w`g]

\place(250,500)[{\scriptstyle \alpha}]

\place(900,500)[\longleftrightarrow]

\square(1300,0)/`@{>}|{\bb}`@{>}|{\bb}`=/[C`B`D`D;`g_*`w`]

\square(1300,500)/=`@{>}|{\bb}`@{>}|{\bb}`/[A`A`C`B;`v`\alpha`]

\place(1550,500)[{\scriptstyle \widehat{\alpha}}]

\efig
$$
and retrocells $ \beta $ are defined to be $ 2 $-cells in $ {\cal{V}}{\it ert}{\mathbb A} $
$$
\bfig\scalefactor{.8}
\square/`@{>}|{\bb}`@{>}|{\bb}`=/[D`A`B`B\rlap{\ .};`y`f_*`]

\square(0,500)/=`@{>}|{\bb}`@{>}|{\bb}`/[C`C`D`A;`g_*`x`]

\place(250,500)[{\scriptstyle \beta}]

\efig
$$
Then our claimed bijection is just the usual bijection from bicategory
theory:
$$
\frac{\widehat{\alpha} \colon g_* \bdot v \to w \bdot f_*}
{\beta \colon y \bdot g_* \to f_* \bdot x\rlap{\ .}}
$$
(2) From the previous proposition we have $ f_* $ is left adjoint
to $ h_* $ and $ g_* $ left adjoint to $ k_* $, and again our
bijection follows from the usual bicategory one:
$$
\frac{\widehat{\alpha} \colon g_* \bdot v \to w_* \bdot f_*}
{\gamma \colon v \bdot h_* \to k_* \bdot w}
$$
%
%

\end{proof}

\begin{corollary} (1) If $ f $ has a left adjoint $ h $, $ g $ a left
adjoint $ k $, $ v $ a right adjoint $ x $ and $ w $ a right adjoint
$ y $, then we have a bijection of cells
$$
\bfig\scalefactor{.8}
\square/>`@{>}|{\bb}`@{>}|{\bb}`>/[A`B`C`D;f`v`w`g]

\place(250,250)[{\scriptstyle \alpha}]

\place(900,250)[\longleftrightarrow]

\square(1300,0)/>`@{>}|{\bb}`@{>}|{\bb}`>/[D`C`B`A\rlap{\ .};k`y`x`h]

\place(1550,250)[{\scriptstyle \delta}]

\efig
$$
(2) We get the same bijection if left and right are interchanged in all
four adjunctions.

\end{corollary}

\begin{proof} (1) We have the following bijections
$$
\bfig\scalefactor{.8}
\square/>`@{>}|{\bb}`@{>}|{\bb}`>/[A`B`C`D;f`v`w`g]

\place(250,250)[{\scriptstyle \alpha}]

\place(800,250)[\longleftrightarrow]

\square(1100,0)/>`@{>}|{\bb}`@{>}|{\bb}`>/[B`A`D`C;h`w`v`k]

\morphism(1450,250)/=>/<-200,0>[`;\beta]

\place(1900,250)[\longleftrightarrow]

\square(2200,0)/>`@{>}|{\bb}`@{>}|{\bb}`>/[D`C`B`A;k`y`x`h]

\place(2450,250)[{\scriptstyle \delta}]

\efig
$$
the first by direct application of part (2) of Theorem \ref{Thm-Mates} and
the second by applying part (1) of Theorem \ref{Thm-Mates} in
$ {\mathbb A}^{ret} $ where $ x \dashv v $ and $ y \dashv w $.
Finally $ \delta $ is a cell in $ ({\mathbb A}^{ret})^{ret} \cong {\mathbb A} $.

\noindent (2) For this we use (1) first and then (2) in $ {\mathbb A}^{ret} $
$$
\bfig\scalefactor{.8}
\square/>`@{>}|{\bb}`@{>}|{\bb}`>/[A`B`C`D;f`v`w`g]

\place(250,250)[{\scriptstyle \alpha}]

\place(800,250)[\longleftrightarrow]

\square(1100,0)/>`@{>}|{\bb}`@{>}|{\bb}`>/[C`D`A`B;g`x`y`f]

\morphism(1450,250)/=>/<-200,0>[`;\gamma]

\place(1900,250)[\longleftrightarrow]

\square(2200,0)/>`@{>}|{\bb}`@{>}|{\bb}`>/[D`C`B`A\rlap{\ .};k`y`x`h]

\place(2450,250)[{\scriptstyle \delta}]

\efig
$$

\end{proof}

Note that the statement of the corollary does not refer to
retrocells or companions but it does not seem possible to
prove it directly without companions. The infamous pinwheel
\cite{DawPar93B} pops up in all attempts to do so.

\section{Coretrocells}

There is a dual situation giving two more bijections in the presence
of right adjoints, but the notion of retrocell is not self-dual. In fact
there is a dual notion, coretrocell, which also comes up in practice
as we will see later.

Like for $ 2 $-categories there are duals op and co for double
categories. $ {\mathbb A}^{op} $
has the horizontal direction reversed and $ {\mathbb A}^{co} $ the
vertical. If $ {\mathbb A} $ has companions there is no reason why
$ {\mathbb A}^{op} $ or $ {\mathbb A}^{co} $ should, and even if they
did there is no relation between the retrocells there and those of
$ {\mathbb A} $. Companions in $ {\mathbb A}^{op} $ or
$ {\mathbb A}^{co} $ correspond to conjoints in $ {\mathbb A} $
and we will use these to define coretrocells.

For completeness we recall the notion of conjoint. More details
can be found in \cite{Gra20}.

\begin{definition}

Let $ f \colon A \to B $ be a horizontal arrow in $ {\mathbb A} $.
A {\em conjoint} for $ f $ is a vertical arrow $ v \colon B \tod A $
together with two cells (conjunctions)
$$
\bfig\scalefactor{.8}
\square/>`=`@{>}|{\bb}`=/[A`B`A`A;f``v`]

\place(250,250)[{\scriptstyle \alpha}]

\place(900,250)[\mbox{and}]

\square(1300,0)/=`@{>}|{\bb}`=`>/[B`B`A`B;`v``f]

\place(1550,250)[{\scriptstyle \beta}]

\efig
$$
such that
$$
\bfig\scalefactor{.8}
\square(0,250)/>`=`@{>}|{\bb}`=/[A`B`A`A;f``v`]

\place(250,500)[{\scriptstyle \alpha}]

\square(500,250)/=``=`>/[B`B`A`B;```f]

\place(750,500)[{\scriptstyle \beta}]

\place(1600,500)[= \mbox{\ \ $\id_f$\quad and\quad}]

\square(2200,0)/>`=`@{>}|{\bb}`=/[A`B`A`A;f``v`]

\place(2450,250)[{\scriptstyle \alpha}]

\square(2200,500)/=`@{>}|{\bb}`=`/[B`B`A`B;`v``]

\place(2450,750)[{\scriptstyle \beta}]

\place(3100,500)[= \mbox{\ \ $1_v$}]

\place(3200,0)[.]

\efig
$$

\end{definition}

As we said, this is the vertical dual of the notion of companion
and therefore has the corresponding properties. They are unique
up to globular isomorphism when they exist and we choose
representation that we call $ f^* $. We have $ (g f)^* \cong f^* \bdot g^* $
and $ 1^*_A \cong \id_A $. The choice is arbitrary but in
practice there is a canonical one and for that $ 1^*_A $ is
usually $ \id_A $, which we will assume.

The dual of sliding is {\em flipping}: we have bijections,
natural in every way that makes sense,
$$
\bfig\scalefactor{.9}
\square(0,250)/`@{>}|{\bb}`@{>}|{\bb}`>/<1000,500>[A`C`D`E;`v`w`h]

\place(500,500)[{\scriptstyle \alpha}]

\morphism(0,750)|a|/>/<500,0>[A`B;f]

\morphism(500,750)|a|/>/<500,0>[B`C;g]

\place(1500,500)[\longleftrightarrow]

\square(2100,0)/>``@{>}|{\bb}`>/<500,1000>[B`C`D`E;g``w`h]

\place(2350,500)[{\scriptstyle \beta}]

\morphism(2100,1000)|l|/@{>}|{\bb}/<0,-500>[B`A;f^*]

\morphism(2100,500)|l|/@{>}|{\bb}/<0,-500>[A`D;v]

\efig
$$
and
$$
\bfig\scalefactor{.8}
\square(0,250)/>`@{>}|{\bb}`@{>}|{\bb}`/<1000,500>[A`B`C`E;f`v`w`]

\place(500,500)[{\scriptstyle \alpha}]

\morphism(0,250)|b|/>/<500,0>[C`D;g]

\morphism(500,250)|b|/>/<500,0>[D`E;h]

\place(1500,500)[\longleftrightarrow]

\square(2100,0)/>`@{>}|{\bb}``>/<500,1000>[A`B`C`D\rlap{\ .};f`v``g]

\place(2350,500)[{\scriptstyle \beta}]

\morphism(2600,1000)|r|/@{>}|{\bb}/<0,-500>[B`E;w]

\morphism(2600,500)|r|/@{>}|{\bb}/<0,-500>[E`D;h^*]

\efig
$$

We now complete Proposition \ref{Prop-AdjComp}.

\begin{proposition}

Assuming only those companions and conjoints mentioned,
we have the following natural bijections
$$
\bfig\scalefactor{.67}
\square(0,250)/`=`=`=/<1000,500>[A`A`A`A;```]
\place(500,500)[{\scriptstyle \alpha_1}]
\morphism(0,750)|a|/>/<500,0>[A`B;f]
\morphism(500,750)|a|/>/<500,0>[B`A;h]

\place(1350,500)[\longleftrightarrow]

\square(1700,250)/>`=`@{>}|{\bb}`=/[A`B`A`A;f``h_*`]
\place(1950,500)[{\scriptstyle \alpha_2}]

\place(2650,500)[\longleftrightarrow]

\square(3100,0)/=`=``=/<500,1000>[A`A`A`A;```]
\place(3350,500)[{\scriptstyle \alpha_3}]
\morphism(3600,1000)|r|/@{>}|{\bb}/<0,-500>[A`B;f_*]
\morphism(3600,500)|r|/@{>}|{\bb}/<0,-500>[B`A;h_*]

\morphism(500,100)/<->/<0,-300>[`;]

\morphism(1950,100)/<->/<0,-300>[`;]

\square(250,-1000)/>`@{>}|{\bb}`=`=/[B`A`A`A;h`f^*``]
\place(500,-750)[{\scriptstyle \alpha_4}]

\place(1250,-750)[\longleftrightarrow]

\square(1700,-1000)/=`@{>}|{\bb}`@{>}|{\bb}`=/[B`B`A`A;`f^*`h_*`]
\place(1950,-750)[{\scriptstyle \alpha_5}]

\morphism(500,-1200)/<->/<0,-300>[`;]

\square(250,-2700)/=``=`=/<500,1000>[A`A`A`A;```]
\place(500,-2200)[{\scriptstyle \alpha_6}]
\morphism(250,-1700)|l|/@{>}|{\bb}/<0,-500>[A`B;h^*]
\morphism(250,-2200)|l|/@{>}|{\bb}/<0,-500>[B`A;f^*]

\square(2850,-3950)/=`=`=`/<1000,500>[B`B`B`B;```]
\place(3350,-3700)[{\scriptstyle \beta_1}]
\morphism(2850,-3950)|b|/>/<500,0>[B`A;h]
\morphism(3350,-3950)|b|/>/<500,0>[A`B;f]
\place(2520,-3700)[\longleftrightarrow]
\square(1700,-3950)/=`@{>}|{\bb}`=`>/[B`B`A`B;`h_*``f]
\place(1950,-3700)[{\scriptstyle \beta_2}]
\place(1200,-3700)[\longleftrightarrow]
\square(250,-4200)/=``=`=/<500,1000>[B`B`B`B;```]
\place(500,-3700)[{\scriptstyle \beta_3}]
\morphism(250,-3200)|l|/@{>}|{\bb}/<0,-500>[B`A;h_*]
\morphism(250,-3700)|l|/@{>}|{\bb}/<0,-500>[A`B;f_*]
\morphism(1950,-2900)/<->/<0,-300>[`;]
\morphism(3350,-2900)/<->/<0,-300>[`;]
\square(3100,-2700)/=`=`@{>}|{\bb}`>/[B`B`B`A;``f^*`h]
\place(3350,-2450)[{\scriptstyle \beta_4}]
\place(2650,-2450)[\longleftrightarrow]
\square(1700,-2700)/=`@{>}|{\bb}`@{>}|{\bb}`=/[B`B`A`A;`h_*`f^*`]
\place(1950,-2450)[{\scriptstyle \beta_5}]
\morphism(3350,-1750)/<->/<0,-300>[`;]
\square(3100,-1550)/=`=``=/<500,1000>[B`B`B`B\rlap{\ .};```]
\place(3350,-1050)[{\scriptstyle \beta_6}]
\morphism(3600,-550)|r|/@{>}|{\bb}/<0,-500>[B`A;f^*]
\morphism(3600,-1050)|r|/@{>}|{\bb}/<0,-500>[A`B;h^*]
%
%
\efig
$$
The following are then equivalent.
\begin{itemize}

	\item[(1)] $ h $ is left adjoint to $ f $ with adjunctions $ \alpha_1 $
	and $ \beta_1 $
	
	\item[(2)] $ h_* $ is a conjoint for $ f $ with conjunctions $ \alpha_2$
	and $ \beta_2 $
	
	\item[(3)] $ f_* $ is left adjoint to $ h_* $ with adjunctions $ \alpha_3 $
	and $ \beta_3 $
	
	\item[(4)] $ f^* $ is a companion for $ h $ with binding cells $ \alpha_4 $
	and $ \beta_4 $
	
	\item[(5)] $ f^* $ is isomorphic to $ h_* $ with inverse isomorphisms
	$ \alpha_5 $ and $ \beta_5 $
	
	\item[(6)] $ f^* $ is left adjoint to $ h^* $ with ajdunctions $ \alpha_6 $
	and $ \beta_6 $.

\end{itemize}

\end{proposition}

\begin{definition}
Suppose that in $ {\mathbb A} $ every horizontal arrow $ f $ has
a conjoint $ f^* $, then a {\em coretrocell}
$$
\bfig\scalefactor{.8}
\square/>`@{>}|{\bb}`@{>}|{\bb}`>/[A`B`C`D;f`v`w`g]

\morphism(260,200)|r|/=>/<0,200>[`;\alpha]

\efig
$$
is a (standard) cell
$$
\bfig\scalefactor{.8}
\square/`@{>}|{\bb}`@{>}|{\bb}`=/[D`A`C`C;`g^*`v`]

\place(250,500)[{\scriptstyle \alpha}]

\square(0,500)/=`@{>}|{\bb}`@{>}|{\bb}`/[B`B`D`A;`w`f^*`]

\efig
$$
in $ {\mathbb A} $.

\end{definition}

Coretrocells are retrocells in $ {\mathbb A}^{co} $. So all
properties of retrocells dualize to coretrocells. In particular
we have a double category $ {\mathbb A}^{cor} $ whose cells
are coretrocells. Dualities can be confusing so we list them
here.

\begin{proposition}
(1) If $ {\mathbb A} $ has conjoints then $ {\mathbb A}^{op} $ and
$ {\mathbb A}^{co} $ have companions and

\noindent (a) $ ({\mathbb A}^{cor})^{op} = ({\mathbb A}^{op})^{ret} $

\noindent (b) $ ({\mathbb A}^{cor})^{co} = ({\mathbb A}^{co})^{ret} $

\vspace{2mm}

\noindent (2) If $ {\mathbb A} $ has companions then $ {\mathbb A}^{op} $ and
$ {\mathbb A}^{co} $ have conjoints and

\noindent (a) $ ({\mathbb A}^{ret})^{op} = ({\mathbb A}^{op})^{cor} $

\noindent (b) $ ({\mathbb A}^{ret})^{co} = ({\mathbb A}^{co})^{cor} $

\vspace{2mm}

\noindent (3) Under the above conditions

\noindent (a) $ ({\mathbb A}^{ret})^{coop} = ({\mathbb A}^{coop})^{ret} $

\noindent (b) $ ({\mathbb A}^{cor})^{coop} = ({\mathbb A}^{coop})^{cor} $.

\end{proposition}

Passing between $ {\mathbb A} $ and $ {\mathbb A}^{co} $ switches
left adjoints to right (both horizontal and vertical), switches companions
and conjoints, and retrocells with coretrocells. Thus we get the dual
theorem for mates.

\begin{theorem}
Assume $ {\mathbb A} $ has conjoints.

(1) If $ x $ is right adjoint to $ v $ and $ y $ right adjoint to $ w $,
then there is a bijection between cells $ \alpha $ and coretrocells
$ \beta $
$$
\bfig\scalefactor{.8}
\square/>`@{>}|{\bb}`@{>}|{\bb}`>/[A`B`C`D;f`v`w`g]

\place(250,250)[{\scriptstyle \alpha}]

\place(900,250)[\longleftrightarrow]

\square(1300,0)/>`@{>}|{\bb}`@{>}|{\bb}`>/[C`D`A`B\rlap{\ .};g`x`y`f]

\morphism(1550,200)|r|/=>/<0,200>[`;\beta]

\efig
$$

(2) If $ h $ is right adjoint to $ f $ and $ k $ right adjoint to $ g $,
then there is a bijection between cells $ \alpha $ and coretrocells
$ \gamma $
$$
\bfig\scalefactor{.8}
\square/>`@{>}|{\bb}`@{>}|{\bb}`>/[A`B`C`D;f`v`w`g]

\place(250,250)[{\scriptstyle \alpha}]

\place(900,250)[\longleftrightarrow]

\square(1300,0)/>`@{>}|{\bb}`@{>}|{\bb}`>/[B`A`D`C\rlap{\ .};g`w`v`f]

\morphism(1550,200)|r|/=>/<0,200>[`;\gamma]

\efig
$$

\end{theorem}

Whereas we think of companions as vertical arrows isomorphic
to horizontal ones, it makes sense to think of a cell $ \alpha $
as above as a cell
$$
\bfig\scalefactor{.8}
\square/`@{>}|{\bb}`@{>}|{\bb}`=/[C`B`D`D;`g_*`w`]

\place(250,500)[{\scriptstyle \widehat{\alpha}}]

\square(0,500)/=`@{>}|{\bb}`@{>}|{\bb}`/[A`A`C`B;`v`f_*`]

\efig
$$
(which it corresponds to bijectively) and reversing its direction would
give a natural notion of a cell in the opposite direction, thus giving
retrocells. Coretrocells, on the other hand, are less intuitive. We
think of conjoints as vertical arrows adjoint to horizontal ones, and
although there is a bijection between cells $ \alpha $ and cells
$$
\bfig\scalefactor{.8}
\square/`@{>}|{\bb}`@{>}|{\bb}`=/[A`D`C`C\rlap{\ ,};`v`g^*`]

\place(250,500)[{\scriptstyle \alpha^\vee}]

\square(0,500)/=`@{>}|{\bb}`@{>}|{\bb}`/[B`B`A`D;`f^*`w`]

\efig
$$
this is more in the nature of a proposition than a tautology.
Nevertheless, formally the two bijections are dual, so have
the same status. Reversing the direction of the $ \alpha^\vee $
gives us coretrocells, and they do come up in practice as we will
see in the next sections.

\section{Retrocells for spans and such}

If $ {\bf A} $ is a category with pullbacks, we get a double
category $ {\mathbb S}{\rm pan} {\bf A} $ whose horizontal
part is $ {\bf A} $, whose vertical arrows are spans and whose
cells are span morphisms, modified to account for the
horizontal arrows
$$
\bfig\scalefactor{.8}
\square(0,250)/>`@{>}|{\bb}`@{>}|{\bb}`>/[A`B`C`D;f`S`T`g]

\place(250,500)[{\scriptstyle \alpha}]

\place(900,500)[=]

\square(1300,0)[S`T`C`D\rlap{\ .};\alpha`\sigma_1	`\tau_1`g]

\square(1300,500)/>`<-`<-`/[A`B`S`T;f`\sigma_0`\tau_0`]

\efig
$$
$ {\mathbb S}{\rm pan} {\bf A} $ has companions $ f_* $ and
conjoints $ f^* $:
$$
\bfig\scalefactor{.7}
\place(0,400)[f_*\ \ =]

\morphism(400,400)|r|/>/<0,400>[A`A;1_A]

\morphism(400,400)|r|/>/<0,-400>[A`B;f]

\place(900,400)[\mbox{and}]

\place(1400,400)[f^*\ \ =]

\morphism(1800,400)|r|/>/<0,400>[A`B;f]

\morphism(1800,400)|r|/>/<0,-400>[A`A\rlap{\ \ .};1_A]

\efig
$$

A retrocell $ \beta $ is
$$
\bfig\scalefactor{.8}
\square(0,250)/>`@{>}|{\bb}`@{>}|{\bb}`>/[A`B`C`D;f`S`T`g]

\morphism(350,500)/=>/<-200,0>[`;\beta]

\place(900,500)[=]

\square(1400,0)/>`>`>`=/<650,500>[T \times_B A`S`D`D;\beta`\tau_1 p_1`g \sigma_1`]

\square(1400,500)/=`<-`<-`/<650,500>[A`A`T \times_B A`S;`p_2`\sigma_0`]

\efig
$$
where
$$
\bfig\scalefactor{.8}
\square[T \times_B A`A`T`B;p_2`p_1`f`\tau_0]

\efig
$$
is a pullback.

A coretrocell $ \gamma $ is
$$
\bfig\scalefactor{.8}
\square(0,250)/>`@{>}|{\bb}`@{>}|{\bb}`>/[A`B`C`D;f`S`T`g]

\morphism(250,450)|r|/=>/<0,200>[`;\gamma]

\place(900,500)[=]

\square(1400,0)/>`>`>`=/<550,500>[C \times_D T`S`C`C\rlap{\ .};\gamma`p_1`\sigma_1`]

\square(1400,500)/=`<-`<-`/<550,500>[B`B`C \times_D T`S;`\tau_0 p_2`f \sigma_0`]

\efig
$$

When $ {\bf A} = {\bf Set} $ we can represent an element
$ s \in S $ with $ \sigma_0 s = a $ and $ \sigma_1 s = c $ by
an arrow $ a \todo{s} c $. Then a morphism of spans $ \alpha $
is a function
$$
(a \todo{s} c) \longmapsto (f a \todo{\alpha(s)} g c) .
$$

For a retrocell $ \beta $, an element of $ T \times_B A $ is a pair
$ (b \todo{t} d, a) $ such that $ f a = b $ so we can represent it
as $ f a \todo{t} d $. Then $ \beta $ is a function
$ (f a \todo{t} d) \longmapsto (a \todo{\beta t} \beta_1 t) $ with
$ g \beta_1 t = d $. If we picture $ S $ as lying over $ T $
(thinking of (co)fibrations) then $ \beta $ is a lifting: for every
$ t $ we are given a $ \beta t $
$$
\bfig\scalefactor{.8}
\square/@{>}|{\bb}`--`--`@{>}|{\bb}/[a`\beta_1 t`f a`d;\beta t```t]

\morphism(250,200)/|->/<0,200>[`;]

\morphism(900,0)/--/<0,500>[T`S;]

\efig
$$
So it is like an opfibration but without any of the category
structure around (in particular we cannot say that ``$ \beta t $
is over $ t $'').

For a coretrocell $ \gamma $, an element of $ C \times_D T $ is
a pair $ (c, b \todo{t} d) $ with $ g c = d $ which we can write as
$ b \todo{t} g c $. $ \gamma $ then assigns to such a $ t $ an
$ S $ element $ \gamma_0 t \to^{\gamma t} c $ with
$ f \gamma_0 t = b $, i.e. a lifting from $ T $ to $ S $
$$
\bfig\scalefactor{.8}
\square/@{>}|{\bb}`--`--`@{>}|{\bb}/[\gamma_0 t`c`b`g c\rlap{\ ,};\gamma t```t]

\morphism(250,200)/|->/<0,200>[`;]

\efig
$$
much like a fibration, though without the category structure.

This example shows well the difference between retrocells
and coretrocells and their comparison with actual cells.

The story for relations is much the same. If $ {\bf A} $ is a
regular category and
$$
\bfig\scalefactor{.8}
\square/>`@{>}|{\bb}`@{>}|{\bb}`>/[A`B`C`D;f`R`S`g]

\efig
$$
is a boundary in $ {\mathbb R}{\rm el} {\bf A} $, i.e. $ f $ and $ g $
are morphisms and $ R $ and $ S $ are relations, then in the
internal language of $ A $, there is a (necessarily unique) cell
iff
$$
a \sim_R c \Rightarrow f a \sim_S g c\ ,
$$
there is a retrocell iff
$$
f a \sim_S d \Rightarrow \exists c (a \sim_R c \wedge g c = d)
$$
and a coretrocell iff
$$
b \sim_S g c \quad\Rightarrow\quad \exists a (a \sim_R c \wedge f a = b) .
$$

Profunctors are the relations of the $ \Cat $ world. There is a double
category which we call $ \CCat $ whose objects are small categories,
horizontal arrows functors, vertical arrows profunctors, and cells
the appropriate natural transformations. In a typical cell
$$
\bfig\scalefactor{.8}
\square/>`@{>}|{\bb}`@{>}|{\bb}`>/[{\bf A}`{\bf B}`{\bf C}`{\bf D};F`P`Q`G]

\place(250,250)[{\scriptstyle t}]

\efig
$$
$ t $ is a natural transformation $ P (-, =) \ \to \  Q (F -, G =) $. $ \CCat $
has companions and conjoints:
$$
F_* (A, B) = {\bf B} (FA, B)
$$
$$
F^* (B, A) = {\bf B} (B, FA) .
$$
We denote an element $ p \in P (A, C) $ by an arrow
$ p \colon A \tod C $. So the action of $ t $ is
$$
t \colon (p \colon A \tod C) \longmapsto (t p \colon FA \tod G C)
$$
natural in $ A $ and $ C $, of course.

A retrocell
$$
\bfig\scalefactor{.8}
\square/>`@{>}|{\bb}`@{>}|{\bb}`>/[{\bf A}`{\bf B}`{\bf C}`{\bf D};F`P`Q`G]

\morphism(350,250)/=>/<-200,0>[`;\phi]

\efig
$$
is a natural transformation $ \phi \colon Q \otimes_{\bf B} F_* \to
G_* \otimes_{\bf C} P $. An element of $ Q \otimes_{\bf B} F_* (A, D) $
is an element of $ Q (FA, D) $, $ g \colon FA \tod D $. An element of
$ G_* \otimes_{\bf C} P (A, D) $ is an equivalence class
$$
[p \colon A \tod C, d \colon GC \to D]_C .
$$
So a retrocell assigns to each element of $ Q $, $ q \colon FA \to D $,
an equivalence class
$$
[\phi (q) \colon A \tod C, \ov{\phi} (q) \colon GC \to D] .
$$
We can think of it as a lifting, like for spans
$$
\bfig\scalefactor{.8}
\square/@{>}|{\bb}`--``@{>}|{\bb}/<600,800>[A`C`FA`D\rlap{\ .};\phi(q)```q]

\morphism(600,800)/--/<0,-400>[C`GC;]

\morphism(600,400)/>/<0,-400>[GC`D;]

\morphism(300,250)/|->/<0,400>[`;]

\efig
$$
The lifting $ C $ does not lie over $ D $, there is merely a
comparison $ GC \to D $. Furthermore the lifting is not
unique, but two liftings are connected by a zigzag of
$ {\bf C} $ morphisms. We have not spelled out the details
because we do not know of any occurrences of these
retrocells in print.

Coretrocells of profunctors are similar (dual). We get a
``lifting''
$$
\bfig\scalefactor{.8}
\square/@{>}|{\bb}``--`@{>}|{\bb}/<600,800>[A`C`B`GC\rlap{\ .};p```q]

\morphism(0,800)/--/<0,-400>[A`FA;]

\morphism(0,400)/<-/<0,-400>[FA`B;]

\morphism(300,250)/|->/<0,400>[`;]

\efig
$$

A final variation on the span theme is $ {\bf V} $-matrices. Let
$ {\bf V} $ be a monoidal category with coproducts preserved
by $ \otimes $ in each variable separately. There is associated
a double category which we call $ {\bf V} $-$ {\mathbb S}{\rm et} $. Its
objects are sets and horizontal arrows functions. A vertical arrow
$ A \tod C $ is an $ A \times C $ matrix of objects of $ {\bf V} $,
$ [V_{ac}] $. A cell is a matrix of morphisms
$$
\bfig\scalefactor{.8}
\square/>`@{>}|{\bb}`@{>}|{\bb}`>/[A`B`C`D;f`{[}V_{ac}{]}`{[}W_{bd}{]}`g]

\place(250,250)[{\scriptstyle {[}\alpha_{ac}{]}}]

\efig
$$
$$
\alpha_{ac} \colon V_{ac} \to W_{fa, gc} .
$$
Vertical composition is matrix multiplication
$$
[X_{ce}] \otimes [V_{ac}] = [\sum_{c\in C} X_{ce} \otimes V_{ac}] .
$$
Every horizontal arrow has a companion
$$
f_* = [\Delta_{fa, b}]
$$
and a conjoint
$$
f^* = [\Delta_{b, fa}]
$$
where $ \Delta $ is the ``Kronecker delta''
$$
\Delta_{b, b'} = 
\left\{ \begin{array}{lll}
I & \mbox{if} & b = b'\\
0 & \mbox{if} & b \neq b'.
\end{array}
\right.
$$

A retrocell
$$
\bfig\scalefactor{.8}
\square/>`@{>}|{\bb}`@{>}|{\bb}`>/[A`B`C`D;f`{[}V_{ac}{]}`{[}W_{bd}{]}`g]

\morphism(340,240)/=>/<-200,0>[`;\phi]

\efig
$$
is an $ A \times D $ matrix $ [\phi_{ad}] $
$$
\phi_{ad} \colon W_{fa, d} \to \sum_{gc = d} V_{ac} \ .
$$
A coretrocell
$$
\bfig\scalefactor{.8}
\square/>`@{>}|{\bb}`@{>}|{\bb}`>/[A`B`C`D;f`{[}V_{ac}{]}`{[}W_{bd}{]}`g]

\morphism(250,200)|r|/=>/<0,200>[`;\psi]

\efig
$$
is a $ B \times C $ matrix $ [\psi_{bc}] $
$$
\psi_{bc} \colon W_{b, gc} \to \sum_{fa = b} V_{ac} \ .
$$
For example, if $ {\bf V} = {\bf Ab} $, and we again represent
elements of $ V_{ac} $ by arrows $ a \todo{v} c $ (resp. of
$ W_{bd} $ by $ b \todo{w} d $), then $ \phi $ associates to each
$ f a \todo{w} d $ a finite number of elements $ a \todo{v_i} c_i $
with $ g c_i = d $
$$
\bfig\scalefactor{.8}
\square/@{>}|{\bb}`--`--`@{>}|{\bb}/[a`c_i`f a`d;v_i```w]

\morphism(250,200)/|->/<0,200>[`;]

\place(900,250)[(i = 1, ..., n)\ .]

\efig
$$
Of course the dual situation holds for coretrocells $ \psi $.

So we see that (co)retrocells in each case give liftings but
of a type adapted to the situation. For spans they are
uniquely specified, for relations they exist but are not
specified, for profunctors only up to a connectedness
condition and for matrices of Abelian groups we get a finite
number of them.

\section{Monads}

A monad in $ \Cat $ is a quadruple $ ({\bf A}, T, \eta, \mu) $ where
$ {\bf A} $ is a category, $ T \colon {\bf A} \to {\bf A} $ an endo\-functor,
$ \eta \colon 1_{\bf A} \to T $ and $ \mu \colon T^2 \to T $ natural
transformations satisfying the well-known unit and associativity laws.
In \cite{Str72} Street introduced morphisms of monads
$$
(F, \phi) \colon ({\bf A}, T, \eta, \mu) \to ({\bf B}, S, \kappa, \nu)
$$
as functors $ F \colon {\bf A} \to {\bf B} $ together with a natural
transformation
$$
\bfig\scalefactor{.8}
\square[{\bf A}`{\bf B}`{\bf A}`{\bf B};F`T`S`F]

\morphism(340,320)/=>/<-140,-140>[`;\phi]

\efig
$$
respecting units and multiplications in the obvious way. He
called these monad functors, now called lax monad morphisms
(see \cite{Lei04}). This was done, not just in $ \Cat $, but in a
general $ 2 $-category. Using duality, he also considered
what he called monad opfunctors, i.e. oplax morphisms of
monads, with the $ \phi $ in the opposite direction.

The lax morphisms work well with Eilenberg-Moore algebras,
giving a functor
$$
{\bf EM} (F, \phi) \colon {\bf EM} ({\mathbb T}) \to {\bf EM} ({\mathbb S})
$$
$$
(TA \to^a A) \longmapsto (SFA \to^{\phi A} FTA \to^{Fa} FA)
$$
whereas the oplax ones give functors on the Kleisli categories
$$
{\bf Kl} (F, \psi) \colon {\bf Kl} ({\mathbb T}) \to {\bf Kl} ({\mathbb S})
$$
$$
(A \to^f TB) \longmapsto (FA \to^{Ff}  FTB \to^{\psi B} SFB) .
$$

The story for monads in a double category is this (see \cite{FioGamKoc11, FioGamKoc12}, though
note that there horizontal and vertical are reversed). In general
we just get one kind of morphism, the oplax ones. If we have
companions then we also get the lax ones, and if we also have
conjoints we have another kind. The $ 2 $-category case
considered by Street corresponds to the double category of
coquintets which has companions but not conjoints.

Let $ {\mathbb A} $ be a double category. A vertical {\em monad}
in $ {\mathbb A} $, $ t = (A, t, \eta, \mu) $ consists of an object
$ A $, a vertical endomorphism $ t $ and two cells $ \eta $ and
$ \mu $ as below
$$
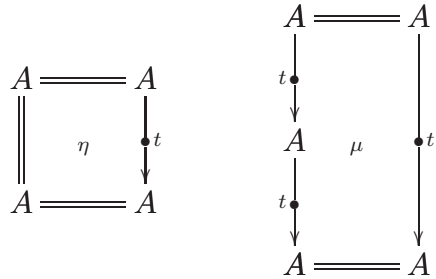
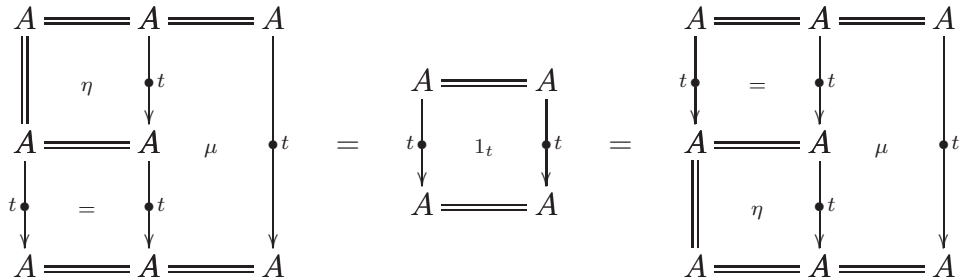
\bfig\scalefactor{.8}
\square(0,250)/=`=`@{>}|{\bb}`=/[A`A`A`A;``t`]

\place(250,500)[{\scriptstyle \eta}]

\square(1100,0)/=``@{>}|{\bb}`=/<500,1000>[A`A`A`A;``t`]

\place(1350,500)[{\scriptstyle \mu}]

\morphism(1100,1000)/@{>}|{\bb}/<0,-500>[A`A;t]

\morphism(1100,500)/@{>}|{\bb}/<0,-500>[A`A;t]

\efig
$$
satisfying
$$
\bfig\scalefactor{.8}
\square/=`@{>}|{\bb}`@{>}|{\bb}`=/[A`A`A`A;`t`t`]

\place(250,250)[{\scriptstyle =}]

\square(0,500)/=`=`@{>}|{\bb}`/[A`A`A`A;``t`]

\place(250,750)[{\scriptstyle \eta}]

\square(500,0)/=``@{>}|{\bb}`=/<500,1000>[A`A`A`A;``t`]

\place(750,500)[{\scriptstyle \mu}]

\place(1300,500)[=]

\square(1600,250)/=`@{>}|{\bb}`@{>}|{\bb}`=/[A`A`A`A;`t`t`]

\place(1850,500)[{\scriptstyle 1_t}]

\place(2400,500)[=]

\square(2700,0)/=`=`@{>}|{\bb}`=/[A`A`A`A;``t`]

\place(2950,250)[{\scriptstyle \eta}]

\square(2700,500)/=`@{>}|{\bb}`@{>}|{\bb}`/[A`A`A`A;`t`t`]

\place(2950,750)[{\scriptstyle =}]

\square(3200,0)/=``@{>}|{\bb}`=/<500,1000>[A`A`A`A;``t`]

\place(3450,500)[{\scriptstyle \mu}]

\efig
$$
$$
\bfig\scalefactor{.8}
\square/=`@{>}|{\bb}`@{>}|{\bb}`=/[A`A`A`A;`t`t`]

\place(250,250)[{\scriptstyle =}]

\square(0,500)/=``@{>}|{\bb}`/<500,1000>[A`A`A`A;``t`]

\place(250,1000)[{\scriptstyle \mu}]

\morphism(0,1500)/@{>}|{\bb}/<0,-500>[A`A;t]

\morphism(0,1000)/@{>}|{\bb}/<0,-500>[A`A;t]

\square(500,0)/=``@{>}|{\bb}`=/<500,1500>[A`A`A`A;``t`]

\place(750,750)[{\scriptstyle \mu}]

\place(1500,750)[=]

\square(2000,0)/=``@{>}|{\bb}`=/<500,1000>[A`A`A`A;``t`]

\place(2250,500)[{\scriptstyle \mu}]

\square(2000,1000)/=`@{>}|{\bb}`@{>}|{\bb}`/[A`A`A`A;`t`t`]

\place(2250,1250)[{\scriptstyle =}]

\morphism(2000,1000)/@{>}|{\bb}/<0,-500>[A`A;t]

\morphism(2000,500)/@{>}|{\bb}/<0,-500>[A`A;t]

\square(2500,0)/=``@{>}|{\bb}`=/<500,1500>[A`A`A`A\rlap{\ .};``t`]

\place(2750,750)[{\scriptstyle \mu}]

\efig
$$

A (horizontal) {\em morphism of monads} $ (f, \psi) \colon (A, t, \eta, \mu)
\to (B, s, \kappa, \nu) $ consists of a horizontal arrow $ f $ and a cell
$ \psi $ as below, such that
$$
\bfig\scalefactor{.8}
\square/=`=`@{>}|{\bb}`=/[A`A`A`A;``t`]

\place(250,250)[{\scriptstyle \eta}]

\square(500,0)/>``@{>}|{\bb}`>/[A`B`A`B;f``s`f]

\place(750,250)[{\scriptstyle \psi}]

\place(1300,250)[=]

\square(1600,0)/>`=`=`>/[A`B`A`B;f```f]

\place(1850,250)[{\scriptstyle \id_f}]

\square(2100,0)/=``@{>}|{\bb}`=/[B`B`B`B;``s`]

\place(2350,250)[{\scriptstyle \kappa}]

\efig
$$
and
$$
\bfig\scalefactor{.8}
\square/=``@{>}|{\bb}`=/<500,1000>[A`A`A`A;``t`]

\place(250,500)[{\scriptstyle \mu}]

\morphism(0,1000)/@{>}|{\bb}/<0,-500>[A`A;t]

\morphism(0,500)/@{>}|{\bb}/<0,-500>[A`A;t]

\square(500,0)/>``@{>}|{\bb}`>/<500,1000>[A`B`A`B;f``s`f]

\place(750,500)[{\scriptstyle \psi}]

\place(1400,500)[=]

\square(1800,0)/>`@{>}|{\bb}`@{>}|{\bb}`>/[A`B`A`B;f`t`s`f]

\place(2050,250)[{\scriptstyle \psi}]

\square(1800,500)/>`@{>}|{\bb}`@{>}|{\bb}`/[A`B`A`B;f`t`s`]

\place(2050,750)[{\scriptstyle \psi}]

\square(2300,0)/=``@{>}|{\bb}`=/<500,1000>[B`B`B`B\rlap{\ .};``s`]

\place(2550,500)[{\scriptstyle \nu}]

\efig
$$
These are the oplax morphisms referred to above.

There are also vertical morphisms of monads, ``bimodules'', whose
composition requires certain well-behaved coequalizers. They are
interesting (see e.g. \cite{Shu08}), of course, but will not concern us here.

If $ {\mathbb A} $ has companions we can also define
retromorphisms of monads. (See \cite{Cla22, DiM22}.)

\begin{definition}
A {\em retromorphism of monads} $ (f, \phi) \colon (A, t, \eta, \mu)
\to (B, s, \kappa, \nu) $ consists of a horizontal arrow $ f $ and a
retrocell $ \phi $
$$
\bfig\scalefactor{.8}
\square/>`@{>}|{\bb}`@{>}|{\bb}`>/[A`B`A`B;f`t`s`f]

\morphism(350,250)/=>/<-200,0>[`;\phi]

\efig
$$
satisfying
$$
\bfig\scalefactor{.8}
\square/=`=`>`=/[B`B`B`B;``s`]

\place(250,250)[{\scriptstyle \kappa}]

\square(0,500)/=`@{>}|{\bb}`@{>}|{\bb}`/[A`A`B`B;`f_*`f_*`]

\place(250,750)[{\scriptstyle =}]

\square(500,0)/``>`=/[B`A`B`B;``f_*`]

\square(500,500)/=``@{>}|{\bb}`/[A`A`B`A;``t`]

\place(750,500)[{\scriptstyle \phi}]

\place(1500,500)[=]

\square(2000,0)/=`@{>}|{\bb}`@{>}|{\bb}`=/[A`A`B`B;`f_*`f_*`]

\place(2250,250)[{\scriptstyle =}]

\square(2000,500)/=`=`@{>}|{\bb}`/[A`A`A`A;``t`]

\place(2250, 750)[{\scriptstyle \eta}]

\efig
$$

$$
\bfig\scalefactor{.8}
\square/=``@{>}|{\bb}`=/<500,1000>[B`B`B`B;``s`]

\place(250,500)[{\scriptstyle \nu}]

\morphism(0,1000)/@{>}|{\bb}/<0,-500>[B`B;s]

\morphism(0,500)/@{>}|{\bb}/<0,-500>[B`B;s]

\square(0,1000)/=`@{>}|{\bb}`@{>}|{\bb}`/[A`A`B`B;`f_*`f_*`]

\place(250,1250)[{\scriptstyle =}]

\morphism(500,1500)/=/<500,0>[A`A;]

\morphism(1000,1500)|r|/@{>}|{\bb}/<0,-1000>[A`A;t]

\morphism(1000,500)|r|/@{>}|{\bb}/<0,-500>[A`B;f_*]

\morphism(500,0)/=/<500,0>[B`B;]

\place(750,750)[{\scriptstyle \phi}]

\place(1300,750)[=]

\square(1700,0)/=`@{>}|{\bb}`@{>}|{\bb}`=/[B`B`B`B;`s`s`]

\place(1950,250)[{\scriptstyle =}]

\square(1700,500)/`@{>}|{\bb}`@{>}|{\bb}`/[B`A`B`B;`s`f_*`]

\square(1700,1000)/=`@{>}|{\bb}`@{>}|{\bb}`/[A`A`B`A;`f_*`t`]

\place(1950,1250)[{\scriptstyle \phi}]

\square(2200,0)/``@{>}|{\bb}`=/[B`A`B`B;``f_*`]

\square(2200,500)/``@{>}|{\bb}`/[A`A`B`A;``t`]

\place(2450,500)[{\scriptstyle \phi}]

\square(2200,1000)/=``@{>}|{\bb}`=/[A`A`A`A;``t`]

\place(2450,1250)[{\scriptstyle =}]

\square(2700,0)/=``@{>}|{\bb}`=/[A`A`B`B;``f_*`]

\square(2700,500)/=``@{>}|{\bb}`/<500,1000>[A`A`A`A;``t`]

\place(2950,1000)[{\scriptstyle \mu}]

\efig
$$

\end{definition}

\begin{proposition}
The identity retrocell is a retromorphism $ (A, t, \eta, \mu) \to
(A, t, \eta, \mu) $. The composite of two retromorphisms of
monads is again one.

\end{proposition}

\begin{proof}
Easy calculation.

\end{proof}

For a monad $ t = (A, t, \eta, \mu) $, Kleisli is a colimit construction,
a universal morphism of the form
$$
(A, t, \eta, \mu) \to (X, \id_X, 1, 1) .
$$
\begin{definition} The {\em Kleisli object} of a vertical monad in a
double category, if it exists, is an object $ Kl(t) $, a horizontal arrow $ f $ and
a cell
$$
\bfig\scalefactor{.8}
\square/>`@{>}|{\bb}`=`>/[A`Kl(t)`A`Kl(t);f`t``f]

\place(250,250)[{\scriptstyle \pi}]

\efig
$$
such that

\noindent (1)
$$
\bfig\scalefactor{.8}
\square/=`=`@{>}|{\bb}`=/[A`A`A`A;``t`]

\place(250,250)[{\scriptstyle \eta}]

\square(500,0)/>``=`>/[A`Kl(t)`A`Kl(t);f```f]

\place(750,250)[{\scriptstyle \pi}]

\place(1500,250)[=]

\square(2000,0)/>`=`=`>/[A`Kl(t)`A`Kl(t);f```f]

\place(2250,250)[{\scriptstyle \id_f}]

\efig
$$

\noindent (2)
$$
\bfig\scalefactor{.8}
\square/=``@{>}|{\bb}`=/<500,1000>[A`A`A`A;``t`]

\place(250,500)[{\scriptstyle \mu}]

\morphism(0,1000)/@{>}|{\bb}/<0,-500>[A`A;t]

\morphism(0,500)/@{>}|{\bb}/<0,-500>[A`A;t]

\square(500,0)/>``=`>/<500,1000>[A`Kl(t)`A`Kl(t);f```f]

\place(750,500)[{\scriptstyle \pi}]

\place(1500,500)[=]

\square(2000,0)/>`@{>}|{\bb}`=`>/[A`Kl(t)`A`Kl(t);f`t``f]

\place(2250,250)[{\scriptstyle \pi}]

\square(2000,500)/>`@{>}|{\bb}`=`/[A`Kl(t)`A`Kl(t);f`t``]

\place(2250,750)[{\scriptstyle \pi}]

\square(2500,0)/=``=`=/<600,1000>[Kl(t)`Kl(t)`Kl(t)`Kl(t);```]

\place(2850,500)[{\scriptstyle \cong}]

\efig
$$
and universal with those properties. That is, for any
$$
\bfig\scalefactor{.8}
\square/>`@{>}|{\bb}`=`>/[A`B`A`B;X`t``X]

\place(250,250)[{\scriptstyle \xi}]

\efig
$$
such that (1) $ \xi \eta = \id $ and (2) $ \xi \mu = \xi \cdot \xi $,
there exists a unique $ h \colon Kl(t) \to B $ such that (1) $ h f = x $
and (2) $ h \pi = \xi $.
\end{definition}

Just by universality, if we have a morphism of monads $ (h, \psi) \colon
(A, t, \eta, \mu) \to (B, s, \kappa, \nu) $ and  the Kleisli objects $ Kl(t) $
and $ Kl(s) $ exist, we get a horizontal arrow $ Kl(h, \psi) $ such that
$$
\bfig\scalefactor{.8}
\square[A`Kl(t)`B`Kl(s)\rlap{\ .};f`h`Kl(h, \psi)`g]

\efig
$$

This does not work for Eilenberg-Moore objects. Asking for a universal
morphism of the form
$$
(X, \id_X, 1, 1) \to (A, t, \eta, \mu)
$$
is not the right thing as can be seen from the usual $ \Cat $ example,
but also in general. For such a morphism $ (u, \theta) $, the unit law
says
$$
\bfig\scalefactor{.8}
\square/=`=`=`=/[X`X`X`X;```]

\place(250,250)[{\scriptstyle 1_{\id_X}}]

\square(500,0)/>``@{>}|{\bb}`>/[X`A`X`A;u``t`u]

\place(750,250)[{\scriptstyle \theta}]

\place(1500,250)[=]

\square(2000,0)/>`=`=`>/[X`A`X`A;u```u]

\place(2250,250)[{\scriptstyle \id_u}]

\square(2500,0)/=``@{>}|{\bb}`=/[A`A`A`A;``t`]

\place(2750,250)[{\scriptstyle \eta}]

\efig
$$
i.e. $ \theta $ must be $ \eta u $ and this is a morphism. Thus monad
morphisms $ (u, \theta) $ are in bijection with horizontal arrows
$ X \to A $. The universal such is $ 1_A $, i.e. we get
$$
\bfig\scalefactor{.8}
\square/>`=`@{>}|{\bb}`>/[A`A`A`A;1_A``t`1_A]

\place(250,250)[{\scriptstyle \eta}]

\efig
$$
not the Eilenberg-Moore object.

\begin{definition} The {\em Eilenberg-Moore object} of a vertical monad
$ (A, t, \eta, \mu) $ is the universal retromorphism
of monads
$$
(X, \id_X, 1, 1) \to^{(u, \theta)} (A, t, \eta, \mu)
$$
$$
\bfig\scalefactor{.8}
\square/>`=`@{>}|{\bb}`>/[X`A`X`A\rlap{\ .};u``t`u]

\morphism(350,250)/=>/<-200,0>[`;\theta]

\efig
$$

\end{definition}

\begin{proposition}

Let $ \cal{A} $ be a $ 2 $-category and $ (A, t, \eta, \mu) $ a monad
in $ \cal{A} $. Then $ (A, t, \eta, \mu) $ is also a monad in the double
category of coquintets $ {\rm co}{\mathbb Q}{\cal{A}} $, and a
retromorphism
$$
(u, \theta) \colon (X, \id_X, 1, 1) \to (A, t, \eta, \mu)
$$
is a $ 1 $-cell $ u \colon X \to A $ and a $ 2 $-cell $ \theta \colon t u \to u $
in $ \cal{A} $ satisfying the unit and associativity laws for a $ t $-algebra.
The universal such is the Eilenberg-Moore object for $ t $.

\end{proposition}

\begin{proof}

This is merely a question of interpreting the definition of
retromorphism in $ {\rm co}{\mathbb Q}{\cal{A}} $.

\end{proof}

We now see immediately how a retrocell
$ (f, \phi) \colon (A, t, \eta, \mu) \to (B, s, \kappa, \nu) $ produces, by
universality, a horizontal arrow
$$
\bfig\scalefactor{.8}
\square<850,500>[EM(t)`EM(s)`A`B\rlap{\ .};EM(f, \phi)`u`u'`f]

\efig
$$

\begin{example}\rm
Let $ \cal{A} $ be a $ 2 $-category and $ {\mathbb Q}{\cal{A}} $ the
double category of quintets in $ \cal{A} $. Recall that a cell in $ {\mathbb Q}{\cal{A}} $
is a quintet in $ \cal{A} $
$$
\bfig\scalefactor{.8}
\square[A`B`C`D\rlap{\ .};f`h`k`g]

\morphism(330,330)/=>/<-140,-140>[`;\alpha]

\efig
$$
Every horizontal arrow $ f $ has a companion, namely $ f $ itself
but viewed as a vertical arrow. A (vertical) monad in $ {\mathbb Q}{\cal{A}} $
is a comonad in $ \cal{A} $. A morphism of monads in $ {\mathbb Q}{\cal{A}} $
is then a lax morphism of comonads, and a retromorphism of monads in
$ {\mathbb Q}{\cal{A}} $ is an oplax morphism of comonads in $ \cal{A} $.

To make the connection with Street's monad functors and opfunctors,
we must take coquintets (the $ \alpha $ in the opposite direction)
$ {\rm co}{\mathbb Q}{\cal{A}} $. Now a monad in $ {\rm co}{\mathbb Q}{\cal{A}} $
is a monad in $ \cal{A} $, a monad morphism in $ {\rm co}{\mathbb Q}{\cal{A}} $
is an oplax morphism of monads, i.e. a monad opfunctor in $ \cal{A} $,
whereas a retromorphism of monads is now a lax morphism of monads,
i.e. a monad functor.

It is unfortunate that the most natural morphisms from a double category
point of view are not the established ones in the literature. At the time of
\cite{Str72}, people were more interested in the Eilenberg-Moore algebras
for a monad as a generalization of Lawvere theories and their algebras,
so it was natural to choose the monad morphisms that worked well
with those, namely lax morphisms, as monad functors. Now, with the
advent of categorical computer science, Kleisli categories have come
into their own, and it is not so clear what the leading concept is, and
double category theory suggests that it may well be the oplax
morphisms.

\end{example}

\begin{example}\rm

Let $ {\bf C} $ be a category with (a choice of) pullbacks. As is
well-known a monad in $ {\mathbb S}{\rm pan} {\bf C} $ is a
category object in $ {\bf C} $. A morphism of monads in
$ {\mathbb S}{\rm pan} {\bf C} $ is an internal functor.

A retromorphism of monads
$$
\bfig\scalefactor{.8}
\square/>`@{>}|{\bb}`@{>}|{\bb}`>/[A_0`B_0`A_0`B_0;F`A_1`B_1`F]

\morphism(350,250)/=>/<-200,0>[`;\phi]

\efig
$$
is first of all a morphism $ F \colon A_0 \to B_0 $ and then a cell
$$
\bfig\scalefactor{.8}
\square/>`>`>`=/<800,500>[B_1 \times_{B_0} A_0`A_0`B_0`B_0;\phi`d_1 p_1`F d_1`]

\square(0,500)/=`<-`<-`/<800,500>[A_0`A_0`B_1 \times_{B_0} A_0`A_0;`p_2`d_0`]

\efig
$$
which must satisfy the unit law
$$
\bfig\scalefactor{.8}
\qtriangle/>`>`>/<850,550>[A_0`B_1\times_{B_0} A_0`A_1;
\langle \id F, 1_{A_0} \rangle`\id`\phi]

\efig
$$
and the composition law
$$
\bfig\scalefactor{.9}
\square/`>``>/<2200,1000>[B_1 \times_{B_0} B_1 \times_{B_0} A_0`B_1 \times_{B_0} A_0 \times_{A_0} A_1
`B_1 \times_{B_0} A_0`A_1\rlap{\ .};`\nu \times_{B_0} A_0``\phi]

\morphism(0,1000)/>/<1100,0>[B_1 \times_{B_0} B_1 \times_{B_0} A_0`B_1 \times_{B_0} A_1;B_1 \times_{B_0} \phi]

\morphism(1100,1000)/>/<1100,0>[B_1 \times_{B_0} A_1`B_1 \times_{B_0} A_0 \times_{A_0} A_1;\cong]

\morphism(2200,1000)|r|/>/<0,-500>[B_1 \times_{B_0} A_0 \times_{A_0} A_1`A_1 \times_{A_0} A_1;\phi \times_{A_0} A_1]

\morphism(2200,500)|r|/>/<0,-500>[A_1 \times_{A_0} A_1`A_1;\mu]

\efig
$$
This is precisely an internal cofunctor \cite{Agu97, Cla20}.

When $ {\bf C} = {\bf Set} $, a cofunctor $ F \colon {\bf A} \tosl {\bf B} $
consists of an object function
$ F \colon {\rm Ob} {\bf A} \to {\rm Ob} {\bf B} $ and a lifting function
$ \phi \colon (b \colon FA \to B) \longmapsto (a \colon A \to A') $
with $ FA' = B $
$$
\bfig\scalefactor{.8}
\square/>`--`--`>/[A`A'`FA`B;a```b]

\morphism(250,180)/|->/<0,200>[`;]

\efig
$$
satisfying

\noindent (1) (unit law) $ \phi (A, 1_{FA}) = 1_A $

\noindent (2) (composition law)
$$
\phi(b'b, A) = \phi (b', A') \phi (b, A) .
$$
So $ F $ is like a split opfibration given algebraically but
without the functor part.

\end{example}

If $ {\mathbb A} $ has conjoints, we can define coretromorphisms
of monads as retromorphisms in $ {\mathbb A}^{op} $ which now
has companions, and monads in $ {\mathbb A}^{op} $ are the same
as monads in $ {\mathbb A} $. Explicitly, an opretromorphism
$$
(f, \theta) \colon (A, t, \eta, \mu) \to (B, s, \kappa, \nu)
$$
consists of a horizontal morphism $ f \colon A \to B $ in $ {\mathbb A} $
and an opretromorphism $ \theta $
$$
\bfig\scalefactor{.8}
\square(0,250)/>`@{>}|{\bb}`@{>}|{\bb}`>/[A`B`A`B;f`t`s`f]

\morphism(250,450)|r|/=>/<0,200>[`;\theta]

\place(800,500)[=]

\square(1100,0)/`@{>}|{\bb}`@{>}|{\bb}`=/[B`A`A`A;`f^*`t`]

\place(1350,500)[{\scriptstyle \theta}]

\square(1100,500)/=`@{>}|{\bb}`@{>}|{\bb}`/[B`B`B`A;`g`f^*`]

\efig
$$
such that
$$
\bfig\scalefactor{.8}
\square/=`@{>}|{\bb}`@{>}|{\bb}`=/[B`B`A`A;`f^*`f^*`]

\place(250,250)[{\scriptstyle =}]

\square(0,500)/=`=`@{>}|{\bb}`/[B`B`B`B;``s`]

\place(250,750)[{\scriptstyle \kappa}]

\square(500,0)/``@{>}|{\bb}`=/[B`A`A`A;``t`]

\place(750,500)[{\scriptstyle \theta}]

\square(500,500)/=``@{>}|{\bb}`/[B`B`B`A;``f^*`]

\place(1500,500)[=]

\square(2000,0)/=`=`@{>}|{\bb}`=/[A`A`A`A;``t`]

\place(2250,250)[{\scriptstyle \eta}]

\square(2000,500)/=`@{>}|{\bb}`@{>}|{\bb}`/[B`B`A`A;`f^*`f^*`]

\place(2250,750)[{\scriptstyle =}]

\efig
$$
and
$$
\bfig\scalefactor{.8}
\square/=`@{>}|{\bb}`@{>}|{\bb}`=/[B`B`A`A;`f^*`f^*`]

\square(0,500)/=``@{>}|{\bb}`/<500,1000>[B`B`B`B;``s`]

\place(250,250)[{\scriptstyle =}]

\place(250,1000)[{\scriptstyle \nu}]

\morphism(0,1500)/@{>}|{\bb}/<0,-500>[B`B;s]

\morphism(0,1000)/@{>}|{\bb}/<0,-500>[B`B;s]

\morphism(500,1500)/=/<500,0>[B`B;]

\morphism(1000,1500)|r|/@{>}|{\bb}/<0,-500>[B`A;f^*]

\morphism(1000,1000)|r|/@{>}|{\bb}/<0,-1000>[A`A;t]

\morphism(500,0)/=/<500,0>[A`A;]

\place(750,750)[{\scriptstyle \theta}]

\place(1300,750)[=]

\square(1600,0)/`@{>}|{\bb}`@{>}|{\bb}`=/[B`A`A`A;`f^*`t`]

\square(1600,500)/=`@{>}|{\bb}`@{>}|{\bb}`/[B`B`B`A;`s`f^*`]

\square(1600,1000)/=`@{>}|{\bb}`@{>}|{\bb}`/[B`B`B`B;`s`s`]

\place(1850,500)[{\scriptstyle \theta}]

\place(1850,1250)[{\scriptstyle =}]

\square(2100,0)/=``@{>}|{\bb}`=/[A`A`A`A;``t`]

\place(2350,250)[{\scriptstyle =}]

\square(2100,500)/``@{>}|{\bb}`/[B`A`A`A;``t`]

\square(2100,1000)/=``@{>}|{\bb}`/[B`B`B`A;``f^*`]

\place(2350,1000)[{\scriptstyle \theta}]

\square(2600,0)/=``@{>}|{\bb}`=/<500,1000>[A`A`A`A\rlap{\ .};``t`]

\place(2850,500)[{\scriptstyle \mu}]

\square(2600,1000)/=``@{>}|{\bb}`/[B`B`A`A;``f^*`]

\place(2850,1250)[{\scriptstyle =}]

\efig
$$

Coretromorphisms do not come up in the formal theory of
monads because the \nobreak{double} category of coquintets of a
$ 2 $-category seldom has conjoints, but $ {\mathbb S}{\rm pan} {\bf C} $
does, and we get opcofunctors, i.e. cofunctors
$ {\bf A}^{op} \to {\bf B}^{op} $. These consist of an object function
$ F \colon {\rm Ob} {\bf A} \to {\rm Ob}{\bf B} $ and a lifting function
$$
\theta \colon (b \colon B \to FA) \longmapsto (a \colon A' \to A)
$$
with $ F A' = A $
$$
\bfig\scalefactor{.8}
\square/>`--`--`>/[A'`A`B`FA;a```b]

\morphism(250,200)/|->/<0,200>[`;]

\efig
$$
satisfying

\noindent (1) $ \theta (A, 1_{FA}) = 1_A $

\noindent (2) $ \theta (A, bb') = \theta (A, b) \theta (A', b') $.

This again illustrates well the difference between retromorphisms
and coretromorphisms and, at the same time, the symmetry of
the concepts. They all move objects forward. Functors move
arrows forward
$$
(a \colon A \to A') \longmapsto (Fa \colon FA \to FA') ,
$$
cofunctors move arrows of the form $ FA \to B $ backward
$$
(b \colon FA \to B) \longmapsto (\phi b \colon A \to A')
$$
and opcofunctors move arrows of the form $ B \to FA $
backward
$$
(b \colon B \to FA) \longmapsto (\theta b \colon A' \to A).
$$

All of this can be extended to the enriched setting for a
monoidal category $ {\bf V} $ which has coproducts
preserved by the tensor in each variable. Then a monad in
$ {\bf V} $-$ {\mathbb S}{\rm et} $ is exactly a small
$ {\bf V} $-category and the retromorphisms are exactly
the enriched cofunctors of Clarke and Di~Meglio
\cite{ClaDim22}, to which we refer the reader for further
details.

\section{Closed double categories}

Many bicategories that come up in practice are closed, i.e.
composition $ \otimes $ has right adjoints in each variable,
$$
Q \otimes (-) \dashv Q \obslash (\ )
$$
$$
(\ )\otimes P \dashv (\ ) \oslash P\rlap{\ .}
$$
Thus we have bijections
\begin{center}
\begin{tabular}{c} 
 $P \to Q \obslash R $ \\[3pt]  \hline \\[-12pt]
 $Q \otimes  P \to R$  \\[3pt] \hline \\[-12pt]
$Q \to R \oslash P $ 
\end{tabular}
\end{center}

We adapt (and adopt) Lambek's notation for the internal homs.
$ \otimes $ is a kind of multiplication and $ \obslash $ and $ \oslash $
divisions.

\begin{example}\rm

The original example in \cite{Lam66}, though not expressed in
bicategorical terms, was $ {\cal{B}} {\it im} $ the bicategory whose
objects are rings, $ 1 $-cells bimodules and $ 2 $-cells linear maps.
Composition is $ \otimes $
$$
\bfig\scalefactor{.8}
\Atriangle/@{<-}|{\bb}`@{>}|{\bb}`@{>}|{\bb}/<400,300>[S`R`T\rlap{\ .};M`N`N\otimes_S M]

\efig
$$
($ M $ is an $ S $-$ R $-bimodule, i.e. left $ S $ - right $ R $ bimodule, etc.)
Given $ P \colon R \tod T $, we have the usual bijections
\begin{center}
\begin{tabular}{c} 
 $N \to P \oslash_R M $\mbox{\quad $T$-$S$ linear} \\[3pt]  \hline \\[-12pt]
 $N \otimes_S M \to P$ \mbox{\quad $T$-$R$ linear} \\[3pt] \hline \\[-12pt]
 $M \to N \obslash_T P$\mbox{\quad $S$-$R$ linear}  
\end{tabular}
\end{center}
where
$$ P \oslash_R M = \Hom_R (M, P) $$
$$ N \obslash_T P = \Hom_T (N, P) $$
are the hom bimodules of $ R $-linear (resp. $ T $-linear) maps.

\end{example}

\begin{example}\rm

The bicategory of small categories and profunctors is closed. For profunctors
$$
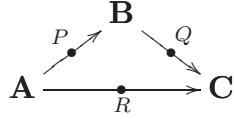
\bfig\scalefactor{.8}
\Atriangle/@{<-}|{\bb}`@{>}|{\bb}`@{>}|{\bb}/<400,300>[{\bf B}`{\bf A}`{\bf C};P`Q`R]

\efig
$$
we have
$$
(Q \obslash_{\bf C} R) (A, B) = \{n.t. \ Q(B, -) \to R (A, -)\}
$$
and
$$
(R \oslash_{\bf A} P) (B, C) = \{n.t. \ P(-, B) \to R(-, C)\}\rlap{\ .}
$$

\end{example}

\begin{example}\rm 
If $ {\bf A} $ has finite limits, then it is locally cartesian closed if and only if
the bicategory of spans in $ {\bf A} $, $ {\cal{S}}{\it pan} {\bf A} $,
is closed (Day \cite{Day74}).

For spans $ A \to/<-/^{p_0}  R \to^{p_1} B $ and
$ B \to/<-/^{\tau_0} T \to^{\tau_1} C $, the composite is given by
the pullback $ T \times_B R $, which we could compute as
the pullback $ P $ below and then composing with $ \tau_1 $
$$
\bfig\scalefactor{.8}
\square/<-`>`>`<-/<700,500>[R`P`A \times B`A \times T;```A \times \tau_0]

\morphism(700,0)|b|/>/<600,0>[A \times T`A \times C;A \times \tau_1]

\place(350,270)[\mbox{$\scriptstyle PB$}]

\efig
$$
i.e. $ T \otimes_B (\ ) $ is the composite
$$
{\bf A}/(A \times B) \to^{(A \times \tau_0)^*} {\bf A}/(A \times T)
\to^{\sum_{A \times \tau_1}} {\bf A}/(A \times C)\ .
$$
$ \sum_{A \times \tau_1} $ always has a right adjoint $ (A \times \tau_1)^* $
and if $ {\bf A} $ is locally cartesian closed so will $ (A \times \tau_0)^* $,
namely $ \prod_{A \times \tau_0} $. So, for $ A \to/<-/^{\sigma_0} S 
\to^{\sigma_1} C $,
$$ T \obslash_C S = \prod_{A \times \tau_0} (A \times \tau_1)^* S\rlap{\ .}
$$
If we interpret this for $ {\bf A} = {\bf Set} $, in terms of fibers
$$
(T \obslash_C S)_{ab} = \prod_c S_{ac}^{T_{bc}} \ .
$$

The situation for $ \oslash_A $ is similar
$$
(S \oslash_A R)_{bc} = \prod_a S_{ac}^{R_{ab}} \ .
$$

\end{example}

These bicategories, and in fact most bicategories that occur in
practice, are the vertical bicategories of naturally occurring double
categories. So a definition of a (vertically) closed double category
would seem in order. And indeed Shulman in \cite{Shu08} did
give one. A double category is closed if its vertical bicategory is.
This definition was taken up by Koudenburg \cite{Kou14} in
his work on pointwise Kan extensions. But both were working
with ``equipments'', double categories with companions and
conjoints. Something more is needed for general double
categories.

\begin{definition}
(Shulman) $ {\mathbb A} $ has {\em globular left homs} if
for every $ y $, $ y \bdot (\ ) $ has a right adjoint $ y \bsd (\ ) $
in $ {\cal{V}}{\it ert} {\mathbb A} $.

\end{definition}

Thus for every $ z $ we have a bijection
$$
\frac{y \bdot x \to z}{x \to y \bsd z} \mbox{\quad\quad in $ {\cal{V}}{\it ert} {\mathbb A} $}
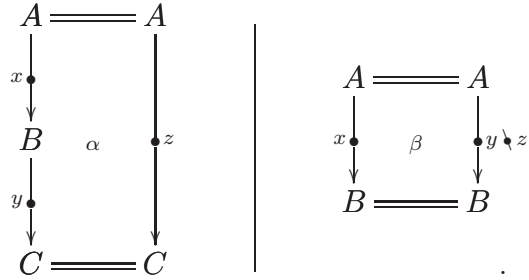
$$

$$
\bfig\scalefactor{.8}
\square/=``@{>}|{\bb}`=/<500,1000>[A`A`C`C;``z`]

\place(250,500)[{\scriptstyle \alpha}]

\morphism(0,1000)/@{>}|{\bb}/<0,-500>[A`B;x]

\morphism(0,500)/@{>}|{\bb}/<0,-500>[B`C;y]

\morphism(900,-20)/-/<0,1060>[`;]

\square(1300,250)/=`@{>}|{\bb}`@{>}|{\bb}`=/[A`A`B`B;`x`y \bsd z`]

\place(1550,500)[{\scriptstyle \beta}]

\place(1900,0)[.]

\efig
$$
Of course there is the usual naturality condition on $ x $,
which is guaranteed by expressing the above bijection as
composition with an evaluation cell $ \epsilon \colon y \bdot
(y \bsd z) \to z $
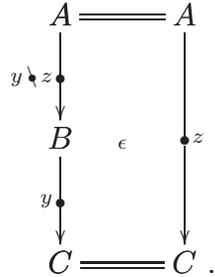
$$
\bfig\scalefactor{.8}
\square/=``@{>}|{\bb}`=/<500,1000>[A`A`C`C\rlap{\ .};``z`]

\place(250,500)[{\scriptstyle \epsilon}]

\morphism(0,1000)/@{>}|{\bb}/<0,-500>[A`B;y \bsd z]

\morphism(0,500)/@{>}|{\bb}/<0,-500>[B`C;y]

\efig
$$
The universal property is then: for every $ \alpha $ there is a
unique $ \beta $, as below, such that
$$
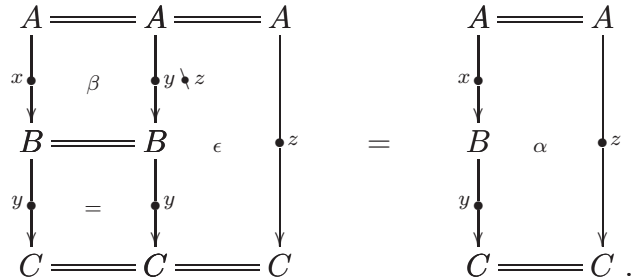
\bfig\scalefactor{.8}
\square/=`@{>}|{\bb}`@{>}|{\bb}`=/[B`B`C`C;`y`y`]

\place(250,250)[{\scriptstyle =}]

\square(0,500)/=`@{>}|{\bb}`@{>}|{\bb}`/[A`A`B`B;`x`y \bsd z`]

\place(250,750)[{\scriptstyle \beta}]

\square(500,0)/=``@{>}|{\bb}`=/<500,1000>[A`A`C`C;``z`]

\place(750,500)[{\scriptstyle \epsilon}]

\place(1400,500)[=]

\square(1800,0)/=``@{>}|{\bb}`=/<500,1000>[A`A`C`C\rlap{\ .};``z`]

\place(2050,500)[{\scriptstyle \alpha}]

\morphism(1800,1000)/@{>}|{\bb}/<0,-500>[A`B;x]

\morphism(1800,500)/@{>}|{\bb}/<0,-500>[B`C;y]

\efig
$$
This shows clearly that $ \bsd $ has nothing to do with
horizontal arrows, and the interplay between the horizontal
and vertical is at the very heart of double categories.

\begin{definition}
$ {\mathbb A} $ has {\em strong left homs (is left closed)} if for every $ y $ and
$ z $ as below there is a vertical arrow $ y \bsd z $ and an
evaluation cell $ \epsilon $ such that for every $ \alpha $ there is
a unique $ \beta $ such that
$$
\bfig\scalefactor{.8}
\square/=`@{>}|{\bb}`@{>}|{\bb}`=/[B`B`C`C;`y`y`]

\place(250,250)[{\scriptstyle =}]

\square(0,500)/>`@{>}|{\bb}`@{>}|{\bb}`/[A`A`B`B;f`x`y \bsd z`]

\place(250,750)[{\scriptstyle \beta}]

\square(500,0)/=``@{>}|{\bb}`=/<500,1000>[A`A`C`C;``z`]

\place(750,500)[{\scriptstyle \epsilon}]

\place(1400,500)[=]

\square(1800,0)/>``@{>}|{\bb}`=/<500,1000>[A'`A`C`C\rlap{\ .};f``z`]

\place(2050,500)[{\scriptstyle \alpha}]

\morphism(1800,1000)/@{>}|{\bb}/<0,-500>[A'`B;x]

\morphism(1800,500)/@{>}|{\bb}/<0,-500>[B`C;y]

\efig
$$

\end{definition}

\begin{proposition}

If $ {\mathbb A} $ has companions and has globular left homs, then
the strong universal property is equivalent to stability under companions:
for every $ f $, the canonical morphism
$$
(y \bsd z) \bdot f_* \to y \bsd (z \bdot f_*)
$$
is an isomorphism.

\end{proposition}

\begin{proof}
(Sketch) For every $ f $ and $ x $ as below we have the
following natural bijections of cells
$$
\bfig\scalefactor{.8}
\square/>``@{>}|{\bb}`=/<500,1000>[A'`A`C`C;f``z`]

\place(250,500)[{\scriptstyle \alpha}]

\morphism(0,1000)/@{>}|{\bb}/<0,-500>[A'`B;x]

\morphism(0,500)/@{>}|{\bb}/<0,-500>[B`C;y]

\morphism(900,-20)/-/<0,1060>[`;]

\square(1300,0)/`@{>}|{\bb}`@{>}|{\bb}`=/[B`A`C`C;`y`z`]

\place(1550,500)[{\scriptstyle \ov{\alpha}}]

\square(1300,500)/=`@{>}|{\bb}`@{>}|{\bb}`/[A'`A'`B`A;`x`f_*`]

\morphism(2200,-20)/-/<0,1060>[`;]

\square(2600,250)/=`@{>}|{\bb}`@{>}|{\bb}`=/[A'`A'`B`B;`x`y \bsd (z \bdot f_*)`]

\place(2850,500)[{\scriptstyle \beta}]

\place(3500,0)[.]

\efig
$$
$ y \bsd z $ is strong iff we have the following bijections
$$
\bfig\scalefactor{.8}
\square/>``@{>}|{\bb}`=/<500,1000>[A'`A`C`C;f``z`]

\place(250,500)[{\scriptstyle \alpha}]

\morphism(0,1000)/@{>}|{\bb}/<0,-500>[A'`B;x]

\morphism(0,500)/@{>}|{\bb}/<0,-500>[B`C;y]

\morphism(900,-20)/-/<0,1060>[`;]

\square(1300,250)/>`@{>}|{\bb}`@{>}|{\bb}`=/[A'`A`B`B;f`x`y \bsd z`]

\place(1550,500)[{\scriptstyle \gamma}]

\morphism(2200,-20)/-/<0,1060>[`;]

\square(2600,0)/=`@{>}|{\bb}``=/<500,1000>[A'`A'`B`B;`x``]

\place(2850,500)[{\scriptstyle \ov{\gamma}}]

\morphism(3100,1000)|r|/@{>}|{\bb}/<0,-500>[A'`A;f_*]

\morphism(3100,500)|r|/@{>}|{\bb}/<0,-500>[A`B\rlap{\ .};y \bsd z]

\efig
$$

\end{proof}

\begin{proposition}
If $ {\mathbb A} $ has conjoints, then the strong universal
property is equivalent to the globular one.

\end{proposition}

\begin{proof}
(Sketch) For every $ f $ and $ x $ as below we have the
following natural bijections
$$
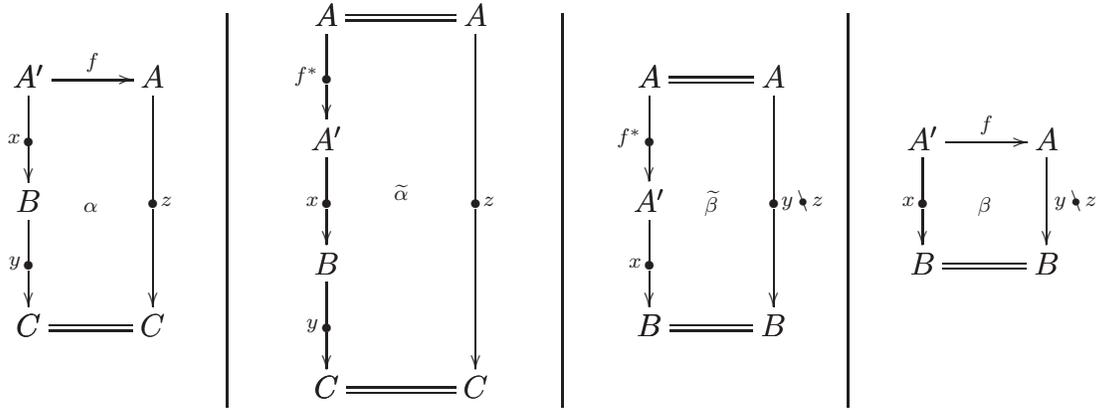
\bfig\scalefactor{.8}
\square/>``@{>}|{\bb}`=/<500,1000>[A'`A`C`C;f``z`]

\place(250,500)[{\scriptstyle \alpha}]

\morphism(0,1000)/@{>}|{\bb}/<0,-500>[A'`B;x]

\morphism(0,500)/@{>}|{\bb}/<0,-500>[B`C;y]

\morphism(800,-320)/-/<0,1650>[`;]

\square(1200,-250)/=``@{>}|{\bb}`=/<600,1500>[A`A`C`C;``z`]

\place(1500,550)[{\scriptstyle \widetilde{\alpha}}]

\morphism(1200,1250)/@{>}|{\bb}/<0,-500>[A`A';f^*]

\morphism(1200,750)/@{>}|{\bb}/<0,-500>[A'`B;x]

\morphism(1200,250)/@{>}|{\bb}/<0,-500>[B`C;y]

\morphism(2150,-320)/-/<0,1650>[`;]

\square(2500,0)/=``@{>}|{\bb}`=/<500,1000>[A`A`B`B;``y \bsd z`]

\morphism(2500,1000)/@{>}|{\bb}/<0,-500>[A`A';f^*]

\morphism(2500,500)/@{>}|{\bb}/<0,-500>[A'`B;x]

\place(2750,500)[{\scriptstyle \widetilde{\beta}}]

\morphism(3300,-320)/-/<0,1650>[`;]

\square(3600,250)/>`@{>}|{\bb}`>`=/[A'`A`B`B;f`x`y \bsd z`]

\place(3850,500)[{\scriptstyle \beta}]

\efig
$$

\end{proof}

All of the examples above have conjoints so the left homs are
automatically strong.

\vspace{2mm}

Of course, $ y \bsd z $ is functorial in $ y $ and $ z $, contravariant
in $ y $ and covariant in $ z $, but only for globular cells
$ \beta $, $ \gamma $
$$
y' \to^\beta y \quad \& \quad z \to^\gamma z' \quad \leadsto \quad y \bsd z \to^{\beta \bsd \gamma}
y' \bsd z' \ .
$$
For general double category cells $ \beta $, $ \gamma $
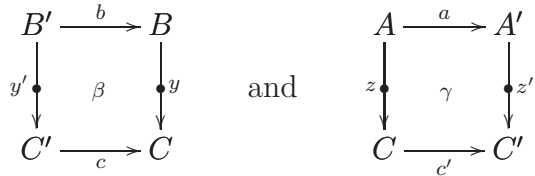
$$
\bfig\scalefactor{.8}
\square/>`@{>}|{\bb}`@{>}|{\bb}`>/[B'`B`C'`C;b`y'`y`c]

\place(250,250)[{\scriptstyle \beta}]

\place(950,250)[\mbox{and}]

\square(1400,0)/>`@{>}|{\bb}`@{>}|{\bb}`>/[A`A'`C`C';a`z`z'`c']

\place(1650,250)[{\scriptstyle \gamma}]

\efig
$$
we would hope to get a cell
$$
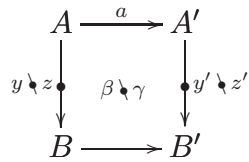
\bfig\scalefactor{.8}
\square/>`@{>}|{\bb}`@{>}|{\bb}`>/[A`A'`B`B';a`y \bsd z`y' \bsd z'`]

\place(250,250)[{\scriptstyle \beta \bsd \gamma}]

\efig
$$
but $ b $ is in the wrong direction, and there are $ c $ and $ c' $ in
opposite directions.  If we reverse $ b $ and $ c $ then $ \beta $ is
in the wrong direction. That was the motivation for retrocells.

\begin{proposition}
Suppose $ {\mathbb A} $ has companions and is (strongly) left
closed. Then a retrocell $ \beta $ and a standard cell $ \gamma $
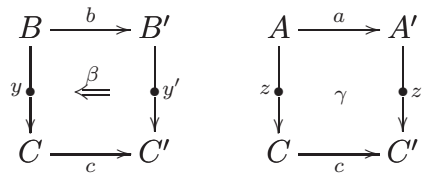
$$
\bfig\scalefactor{.8}
\square/>`@{>}|{\bb}`@{>}|{\bb}`>/[B`B'`C`C';b`y`y'`c]

\morphism(350,250)/=>/<-200,0>[`;\beta]

\square(1000,0)/>`@{>}|{\bb}`@{>}|{\bb}`>/[A`A'`C`C';a`z`z'`c]

\place(1250,250)[{\scriptstyle \gamma}]

\efig
$$
induce a canonical cell
$$
\bfig\scalefactor{.8}
\square/>`@{>}|{\bb}`@{>}|{\bb}`>/[A`A'`B`B'\rlap{\ .};a`y \bsd z`y' \bsd z'`b]

\place(250,250)[{\scriptstyle \beta \bsd \gamma}]

\efig
$$

\end{proposition}

\begin{proof}
(Sketch) A candidate $ \xi $ for $ \beta \bsd \gamma $ would satisfy
the following bijections
$$
\bfig\scalefactor{.8}
\square(0,250)/>`@{>}|{\bb}`@{>}|{\bb}`>/[A`A'`B`B';a`y \bsd z`y' \bsd z'`b]

\place(250,500)[{\scriptstyle \xi}]

\morphism(950,-120)/-/<0,1250>[`;]

\square(1400,0)/>``@{>}|{\bb}`=/<500,1000>[A`A'`B'`B';a``y' \bsd z'`]

\place(1650,500)[{\scriptstyle \ov{\xi}}]

\morphism(1400,1000)/@{>}|{\bb}/<0,-500>[A`B;y \bsd z]

\morphism(1400,500)/@{>}|{\bb}/<0,-500>[B`B';b_*]

\morphism(2300,-320)/-/<0,1650>[`;]

\square(2800,-250)/>``@{>}|{\bb}`=/<600,1500>[A`A'`C'`C';a``z'`]

\place(3100,500)[{\scriptstyle \ov{\ov{\xi}} }]

\morphism(2800,1250)/@{>}|{\bb}/<0,-500>[A`B;y \bsd z]

\morphism(2800,750)/@{>}|{\bb}/<0,-500>[B`B';b_*]

\morphism(2800,250)/@{>}|{\bb}/<0,-500>[B'`C';y']

\efig
$$
and there is indeed a canonical $ \ov{\ov{\xi}} $, namely
$$
\bfig\scalefactor{.8}
\square/`@{>}|{\bb}`@{>}|{\bb}`=/[B'`C`C'`C';`y'`c_*`]

\square(0,500)/=`@{>}|{\bb}`@{>}|{\bb}`/[B`B`B'`C;`b_*`y`]

\square(0,1000)/=`@{>}|{\bb}`@{>}|{\bb}`/[A`A`B`B;`y \bsd z`y \bsd z`]

\square(500,0)/=``@{>}|{\bb}`=/[C`C`C'`C';``c_*`]

\square(500,500)/=``@{>}|{\bb}`/<500,1000>[A`A`C`C;``z`]

\square(1000,0)/>``=`=/[C`C'`C'`C'\rlap{\ .};c```]

\square(1000,500)/>``@{>}|{\bb}`/<500,1000>[A`A'`C`C';a``z'`]

\place(250,500)[{\scriptstyle \beta}]

\place(250,1250)[{\scriptstyle =}]

\place(750,250)[{\scriptstyle =}]

\place(750,1000)[{\scriptstyle \epsilon}]

\place(1250,250)[{\lrcorner}]

\place(1250,1000)[{\scriptstyle \gamma}]

\efig
$$

\end{proof}

\noindent In fact the cell $ \beta \bsd \gamma $ is not only
canonical but also functorial, i.e. $ (\beta' \beta) \bsd (\gamma' \gamma) =
(\beta' \bsd \gamma') (\beta \bsd \gamma) $. To express this properly
we must define the categories involved. The codomain of $ \bsd $ is
simply $ {\bf A}_1 $, the category whose objects are vertical arrows of
$ {\mathbb A} $ and whose morphisms are (standard) cells. The
domain of $ \bsd $ is the category which, for lack of a better name,
we call $ {\bf TC} ({\mathbb A}) $ (twisted cospans). Its objects are cospans of
vertical arrows and its cells are pairs $ (\beta, \gamma) $
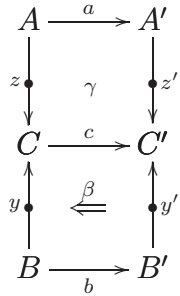
$$
\bfig\scalefactor{.8}
\square/>`@{<-}|{\bb}`@{<-}|{\bb}`>/[C`C'`B`B';c`y`y'`b]

\morphism(340,250)/=>/<-200,0>[`;\beta]

\square(0,500)/>`@{>}|{\bb}`@{>}|{\bb}`/[A`A'`C`C';a`z`z'`]

\place(250,750)[{\scriptstyle \gamma}]

\efig
$$
where $ \beta $ is a retrocell and $ \gamma $ a standard cell. Also
we must flesh out our sketchy construction of $ y \bsd z $. We can express
the universal property of $ y \bsd z $ as representability of a functor
$$
L_{y,z} \colon {\bf A}^{op}_1 \to {\bf Set}\rlap{\ .}
$$
For $ v \colon \ov{A} \tod \ov{B} $, $ L_{y,z} (v) = \{(f, g, \alpha) | f, g, \alpha \mbox{\ \ as in}\ (*)\} $
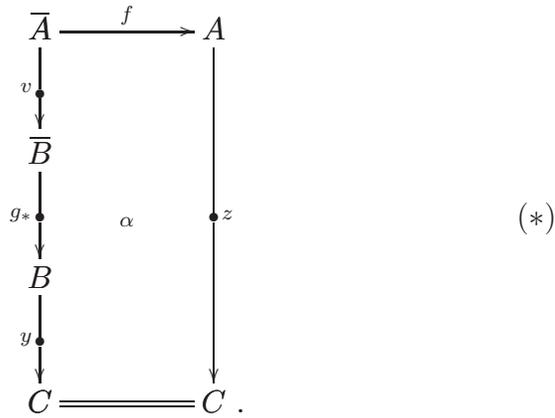
$$
f \colon \ov{A} \to A , g \colon \ov{B} \to B
$$
$$
\bfig\scalefactor{.8}
\square/>``@{>}|{\bb}`=/<700,1500>[\ov{A}`A`C`C\rlap{\ .};f``z`]

\place(-1300,0)[\ ]

\place(350,750)[{\scriptstyle \alpha}]

\morphism(0,1500)/@{>}|{\bb}/<0,-500>[\ov{A}`\ov{B};v]

\morphism(0,1000)/@{>}|{\bb}/<0,-500>[\ov{B}`B;g_*]

\morphism(0,500)/@{>}|{\bb}/<0,-500>[B`C;y]

\place(2000,750)[(*)]

\efig
$$

Some straightforward calculation will show that $ L_{y, z} $
is indeed a functor. The following bijections show that
$ y \bsd z $ is a representing object for $ L_{y, z} $
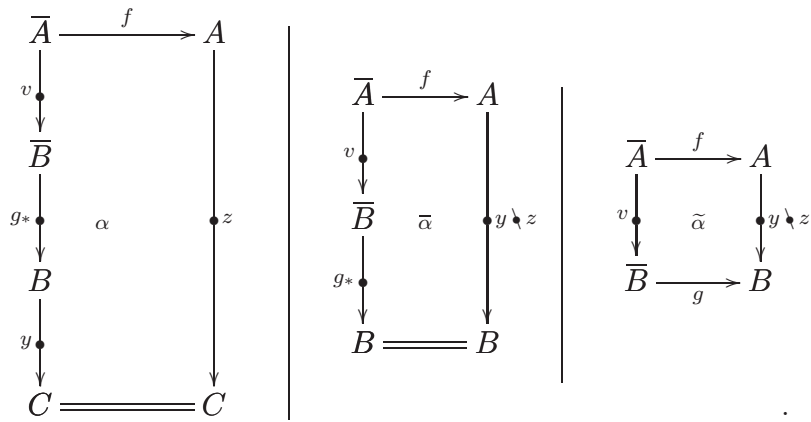
$$
\bfig\scalefactor{.8}
\square/>``@{>}|{\bb}`=/<700,1500>[\ov{A}`A`C`C;f``z`]

\place(250,750)[{\scriptstyle \alpha}]

\morphism(0,1500)/@{>}|{\bb}/<0,-500>[\ov{A}`\ov{B};v]

\morphism(0,1000)/@{>}|{\bb}/<0,-500>[\ov{B}`B;g_*]

\morphism(0,500)/@{>}|{\bb}/<0,-500>[B`C;y]

\morphism(1000,-50)/-/<0,1650>[`;]

\square(1300,250)/>``@{>}|{\bb}`=/<500,1000>[\ov{A}`A`B`B;f``y \bsd z`]

\morphism(1300,1250)/@{>}|{\bb}/<0,-500>[\ov{A}`\ov{B};v]

\morphism(1300,750)/@{>}|{\bb}/<0,-500>[\ov{B}`B;g_*]

\place(1550,750)[{\scriptstyle \ov{\alpha}}]

\morphism(2100,100)/-/<0,1250>[`;]

\square(2400,500)/>`@{>}|{\bb}`@{>}|{\bb}`>/[\ov{A}`A`\ov{B}`B;f`v`y \bsd z`g]

\place(2650,750)[{\scriptstyle{\ovv\alpha}}]

\place(3000,0)[.]

\efig
$$

This gives the full double category universal
property of $ \bsd $ : For every boundary
$$
\bfig\scalefactor{.8}
\square/>`@{>}|{\bb}`@{>}|{\bb}`>/[\ov{A}`A`\ov{B}`B;f`v`y \bsd z`g]

\efig
$$
and $ \alpha $ as below, there exists a unique fill-in $ \beta $ such that
$$
\bfig\scalefactor{.8}
\square/>``@{>}|{\bb}`=/<600,1500>[\ov{A}`A`C`C;f``z`]

\place(300,750)[{\scriptstyle \alpha}]

\morphism(0,1500)/@{>}|{\bb}/<0,-500>[\ov{A}`\ov{B};v]

\morphism(0,1000)/@{>}|{\bb}/<0,-500>[\ov{B}`B;g_*]

\morphism(0,500)/@{>}|{\bb}/<0,-500>[B`C;y]

\place(900,750)[=]

\square(1200,0)/=`@{>}|{\bb}`@{>}|{\bb}`=/[B`B`C`C;`y`y`]

\square(1200,500)/`@{>}|{\bb}`=`/[\ov{B}`B`B`B;`g_*``]

\square(1200,1000)/>`@{>}|{\bb}`@{>}|{\bb}`>/[\ov{A}`A`\ov{B}`B;f`v`y \bsd z`g]

\place(1450,250)[{\scriptstyle =}]

\place(1450,750)[{\lrcorner}]

\place(1450,1250)[{\scriptstyle \beta}]

\square(1700,0)/=``@{>}|{\bb}`=/<600,1500>[A`A`C`C\rlap{\ .};``z`]

\place(2000,750)[{\scriptstyle \epsilon}]

\efig
$$

For $ (\beta, \gamma) $ in $ {\bf TC}({\mathbb A}) $ we get a natural
transformation
$$
\phi_{\beta \gamma} \colon L_{y, z} \to L_{y', z'}
$$

$$
\bfig\scalefactor{.8}
\square(0,250)/>``@{>}|{\bb}`=/<500,1500>[\ov{A}`A`C`C;f``z`]

\place(250,1000)[{\scriptstyle \alpha}]

\morphism(0,1750)/@{>}|{\bb}/<0,-500>[\ov{A}`\ov{B};v]

\morphism(0,1250)/@{>}|{\bb}/<0,-500>[\ov{B}`B;g_*]

\morphism(0,750)/@{>}|{\bb}/<0,-500>[B`C;y]

\place(750,1000)[\longmapsto]

\square(1150,0)/=`@{>}|{\bb}``=/[B'`B'`C'`C';`y'``]

\square(1150,500)/=`@{>}|{\bb}``/<500,1000>[\ov{B}`\ov{B}`B'`B';`(bg)_*``]

\square(1150,1500)/=`@{>}|{\bb}``/[\ov{A}`\ov{A}`\ov{B}`\ov{B};`v``]

\place(1400,250)[{\scriptstyle =}]

\place(1400,1000)[{\scriptstyle \cong}]

\place(1400,1750)[{\scriptstyle =}]

\square(1650,0)/`@{>}|{\bb}``=/[B'`C`C'`C';`y'``]

\square(1650,500)/=`@{>}|{\bb}``/[B`B`B'`C;`b_*``]

\square(1650,1000)/=`@{>}|{\bb}``/[\ov{B}`\ov{B}`B`B;`g_*``]

\square(1650,1500)/=`@{>}|{\bb}``/[\ov{A}`\ov{A}`\ov{B}`\ov{B};`v``]

\place(1900,500)[{\scriptstyle \beta}]

\place(1900,1250)[{\scriptstyle =}]

\place(1900,1750)[{\scriptstyle =}]

\square(2150,0)/=`@{>}|{\bb}``=/[C`C`C'`C';`c_*``]

\square(2150,500)/>``@{>}|{\bb}`/<500,1500>[\ov{A}`A`C`C;f``z`]

\morphism(2150,2000)/@{>}|{\bb}/<0,-500>[\ov{A}`\ov{B};v]

\morphism(2150,1500)/@{>}|{\bb}/<0,-500>[\ov{B}`B;g_*]

\morphism(2150,1000)/@{>}|{\bb}/<0,-500>[B`C;y]

\place(2400,250)[{\scriptstyle =}]

\place(2400,1250)[{\scriptstyle \alpha}]

\square(2650,0)/`@{>}|{\bb}`=`=/[C`C'`C'`C'\rlap{\ .};`c_*``]

\square(2650,500)/>``@{>}|{\bb}`>/<500,1500>[A`A'`C`C';a``z'`c]

\place(2900,250)[{ \lrcorner}]

\place(2900,1250)[{\scriptstyle \gamma}]

\efig
$$

Some calculation is needed to show naturality, which we leave
to the reader. This natural transformation is what gives
$ \beta \bsd \gamma $.

We are now ready for the main theorem of the section.

\begin{theorem}
For $ {\mathbb A} $ a left closed double category with companions,
the internal hom is a functor
$$
\bsd\  \colon {\bf TC} ({\mathbb A}) \to {\bf A}_1\rlap{\ .}
$$

\end{theorem}

\begin{proof}
Let $ (\beta, \gamma) $ and $ (\beta', \gamma') $ be composable
morphisms in $ {\bf TC} ({\mathbb A}) $
$$
\bfig\scalefactor{.8}
\square/>`@{<-}|{\bb}`@{<-}|{\bb}`>/[C`C'`B`B';c`y`y'`b]

\square(0,500)/>`@{>}|{\bb}`@{>}|{\bb}`/[A`A'`C`C';a`z`z'`]

\square(500,0)/>``@{<-}|{\bb}`>/[C'`C''`B'`B''\rlap{\ .};c'``y''`b']

\square(500,500)/>``@{>}|{\bb}`/[A'`A''`C'`C'';a'``z''`]

\morphism(350,250)/=>/<-200,0>[`;\beta]

\place(250,750)[{\scriptstyle \gamma}]

\morphism(850,250)/=>/<-200,0>[`;\beta']

\place(750,750)[{\scriptstyle \gamma'}]

\efig
$$
Then
$$
L_{y,z} \to^{\phi_{\beta, \gamma}} L_{y', z'} \to^{\phi_{\beta', \gamma'}} L_{y'', z''}
$$
takes $ v $ to the composite of 23 cells (most of which are
bookkeeping -- identities, canonical isos, ...) arranged in a $ 5 \times 7 $
array, with 38 objects, and best represented schematically as

 \begin{center}
 \setlength{\unitlength}{.9mm}
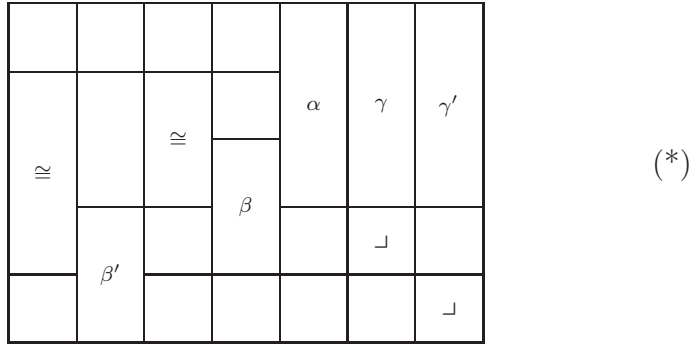
 \begin{picture}(70,50)
 \put(0,0){\framebox(70,50){}}
 
 \put(10,0){\line(0,1){50}}
 \put(20,0){\line(0,1){50}}
 \put(30,0){\line(0,1){50}}
 \put(40,0){\line(0,1){50}}
 \put(50,0){\line(0,1){50}}
 \put(60,0){\line(0,1){50}}

\put(0,10){\line(1,0){10}}
\put(10,20){\line(1,0){20}}
\put(20,10){\line(1,0){50}}

\put(0,40){\line(1,0){40}}
\put(30,30){\line(1,0){10}}
\put(40,20){\line(1,0){30}}

\put(5,25){\makebox(0,0){$\scriptstyle\cong$}}
\put(15,10){\makebox(0,0){$\scriptstyle \beta'$}}
\put(25,30){\makebox(0,0){$\scriptstyle \cong$}}
\put(35,20){\makebox(0,0){$\scriptstyle \beta$}}
\put(45,35){\makebox(0,0){$\scriptstyle \alpha$}}
\put(55,35){\makebox(0,0){$\scriptstyle \gamma$}}
\put(55,15){\makebox(0,0){$ \lrcorner$}}
\put(65,35){\makebox(0,0){$\scriptstyle \gamma'$}}
\put(65,5){\makebox(0,0){$ \lrcorner$}}

\put(95,25){(*)}
 
 \end{picture}
 \end{center}

\noindent whereas
$$
L_{yz} \to^{\phi_{\beta' \beta, \gamma' \gamma}}  L_{y'' z''}
$$
takes $ v $ to

 \begin{center}
 \setlength{\unitlength}{.9mm}
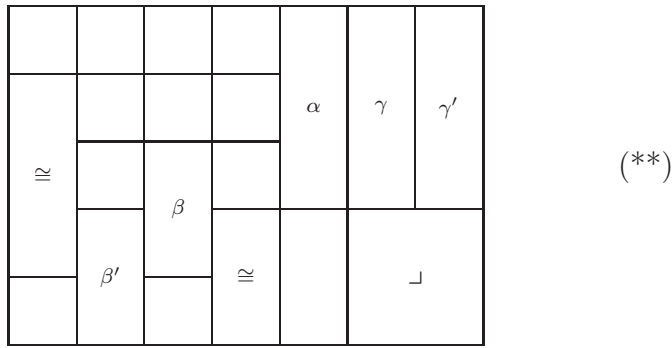
 \begin{picture}(70,50)
 \put(0,0){\framebox(70,50){}}
 
 \put(10,0){\line(0,1){50}}
 \put(20,0){\line(0,1){50}}
 \put(30,0){\line(0,1){50}}
 \put(40,0){\line(0,1){50}}
 \put(50,0){\line(0,1){50}}
 \put(60,20){\line(0,1){30}}

\put(0,10){\line(1,0){10}}
\put(10,20){\line(1,0){10}}
\put(20,10){\line(1,0){10}}
\put(30,20){\line(1,0){40}}
\put(10,30){\line(1,0){30}}
\put(0,40){\line(1,0){40}}

\put(5,25){\makebox(0,0){$\scriptstyle\cong$}}
\put(15,10){\makebox(0,0){$\scriptstyle \beta'$}}
\put(25,20){\makebox(0,0){$\scriptstyle \beta$}}
\put(35,10){\makebox(0,0){$\scriptstyle \cong$}}
\put(45,35){\makebox(0,0){$\scriptstyle \alpha$}}
\put(55,35){\makebox(0,0){$\scriptstyle \gamma$}}
\put(65,35){\makebox(0,0){$\scriptstyle \gamma'$}}
\put(60,10){\makebox(0,0){$ \lrcorner$}}
\put(90,25){(**)}
 
 \end{picture}
 \end{center}

\noindent The three bottom right cells of (**) compose to
the $ 2 \times 2 $ block on the bottom right of (*), so the
$ 5 \times 3 $ part on the right of (*) is equal to the $ 5 \times 4 $
part on the right of (**). And the rest are equal too by
coherence. It follows that
$$
(\beta' \bsd \gamma') (\beta \bsd \gamma) = (\beta' \beta) \bsd (\gamma' \gamma) .
$$
For identities $ 1_{y \bsd z} = 1_y \bsd 1_z $.

\end{proof}

Right closure is dual but the duality is op, switching the direction of
vertical arrows which switches companions with conjoints and retrocells
with coretrocells. We outline the changes.

\begin{definition}
(Shulman) $ {\mathbb A} $ has {\em globular right homs} if
for every $ x $, $ (\ ) \bdot x $ has a right adjoint $ (\ ) \slashdot x $
in $ {\cal{V}}{\it ert} {\mathbb A} $,
$$
\frac{y \bdot x \to z}{y \to z \slashdot x} \mbox{\quad in \ ${\cal{V}}{\it ert}{\mathbb A} $}.
$$
This bijection is mediated by an evaluation cell
$$
\bfig\scalefactor{.8}
\square/=``@{>}|{\bb}`=/<500,1000>[A`A`C`C\rlap{\ .};```]

\morphism(0,1000)/@{>}|{\bb}/<0,-500>[A`B;x]

\morphism(0,500)/@{>}|{\bb}/<0,-500>[B`C;z \slashdot x]

\place(250,500)[{\scriptstyle \epsilon'}]

\efig
$$
The right homs are {\em strong} if $ z \slashdot x $ has the universal property for cells
of the form
$$
\bfig\scalefactor{.8}
\square/=``@{>}|{\bb}`>/<500,1000>[A`A`C'`C\rlap{\ .};``z`g]

\morphism(0,1000)/@{>}|{\bb}/<0,-500>[A`B;x]

\morphism(0,500)/@{>}|{\bb}/<0,-500>[B`C';y]

\place(250,500)[{\scriptstyle \alpha}]

\efig
$$

\end{definition}

\begin{proposition}
If $ {\mathbb A} $ has conjoints and globular right homs, then
the strong universal property is equivalent to the canonical
morphism
$$
g^* \bdot (z \slashdot x) \to (g^* \bdot z) \slashdot x
$$
being an isomorphism. If instead $ {\mathbb A} $ has companions,
then strong is equivalent to globular.

\end{proposition}

Finally, if $ {\mathbb A} $ has conjoints, $ z \slashdot x $ is functorial
in $ z $ and $ x $, for standard cells in $ z $ and for coretrocells in $ x $.
More precisely, $ \slashdot $ is defined on the category $ {\bf TS} ({\mathbb A}) $ whose
objects are spans of vertical arrows, $ (x, z) $, as below, and whose
morphisms are pairs of cells
$$
\bfig\scalefactor{.8}
\square/>`@{>}|{\bb}`@{>}|{\bb}`>/[A`A'`C`C';a`z`z'`c]

\place(250,250)[{\scriptstyle \gamma}]

\square(0,500)/>`@{<-}|{\bb}`@{<-}|{\bb}`/[B`B'`A`A';b`x`x'`]

\morphism(250,850)/=>/<0,-200>[`;\alpha]

\efig
$$
where $ \alpha $ is a coretrocell and $ \gamma $ a standard one.

\begin{theorem}
If $ {\mathbb A} $ has conjoints and is right closed, then
$ \slashdot $ is a functor
$$
\slashdot \ \colon {\bf TS} ({\mathbb A}) \to {\bf A}_1 .
$$

\end{theorem}

For completeness sake, we end this section with a definition.

\begin{definition}
A double category $ {\mathbb A} $ is closed if it is right
closed and left closed.

\end{definition}

\section{A triple category}

As mentioned in the introduction, one of the inspirations for
retrocells was the commuter cells of \cite{GraPar08}.

\begin{definition}
Let $ {\mathbb A} $ be a double category with companions. A cell
$$
\bfig\scalefactor{.8}
\square/>`@{>}|{\bb}`@{>}|{\bb}`>/[A`B`C`D;f`v`w`g]

\place(250,250)[{\scriptstyle \alpha}]

\efig
$$
is a {\em commuter} cell if the associated globular cell
$ \widehat{\alpha} $
$$
\bfig\scalefactor{.8}
\square/`>`=`=/[C`D`D`D;`g_*``]

\square(0,500)/>`@{>}|{\bb}`@{>}|{\bb}`>/[B`B`C`D;f`v`w`g]

\square(0,1000)/=`=`>`/[A`A`B`B;``f_*`]

\place(250,250)[{ \lrcorner}]

\place(250,750)[{\scriptstyle \alpha}]

\place(250,1250)[{ \ulcorner}]

\efig
$$
is a horizontal isomorphism.

\end{definition}

The intent is that the cell $ \alpha $ itself is an isomorphism
making the square commute (up to isomorphism). 

The inverse of $ \widehat{\alpha} $ is a retrocell, so the question
is, how do we express that a cell and a retrocell are inverse to
each other?

Cells and retrocells form a double category (and
ultimately a triple category). For a double category with companions $ {\mathbb A} $, we
define a new (vertical arrow) double category $ {\mathbb V}{\rm ar}
({\mathbb A}) $ as follows. Its objects are the vertical arrows of
$ {\mathbb A} $, its horizontal arrows are standard cells of
$ {\mathbb A} $, and its vertical arrows are retrocells. It is a
thin double category with a unique cell
$$
\bfig\scalefactor{.8}
\square/>`@{>}|{\bb}`@{>}|{\bb}`>/[v`w`v'`w';\alpha`\beta`\gamma`\alpha']

\place(250,250)[{\scriptstyle !}]

\efig
$$

$$
\bfig\scalefactor{.9}
\node a(300,0)[C']
\node b(800,0)[D']
\node c(0,250)[A']
\node d(300,500)[C]
\node e(800,500)[D]
\node f(0,800)[A]
\node g(500,800)[B]

\arrow|b|/>/[a`b;g']
\arrow|l|/@{>}|{\bb}/[c`a;v']
\arrow|r|/>/[d`a;c]
\arrow|r|/>/[e`b;d]
\arrow|b|/>/[d`e;g]
\arrow|l|/>/[f`c;a]
\arrow|l|/@{>}|{\bb}/[f`d;v]
\arrow|r|/@{>}|{\bb}/[g`e;w]
\arrow|a|/>/[f`g;f]

\place(430,650)[{\scriptstyle \alpha}]

\morphism(160,320)|l|/=>/<0,200>[`;\beta]

\node f'(1500,800)[A]
\node g'(2000,800)[B]
\node c'(1500,300)[A']
\node d'(2000,300)[B']
\node a'(1800,0)[C']
\node b'(2300,0)[D']
\node e'(2300,400)[D]

\arrow|a|/>/[f'`g';f]
\arrow|l|/>/[f'`c';a]
\arrow|l|/>/[g'`d';b]
\arrow|r|/@{>}|{\bb}/[g'`e';w]
\arrow|a|/>/[c'`d';f']
\arrow|l|/@{>}|{\bb}/[c'`a';v']
\arrow|b|/>/[a'`b';g']
\arrow|r|/@{>}|{\bb}/[d'`b';w']
\arrow|r|/>/[e'`b';d]

\place(1950,150)[{\scriptstyle \alpha'}]

\morphism(2150,320)|l|/=>/<0,200>[`;\gamma]

\efig
$$
if we have
$$
f' a = b f
$$
$$
g' c = d g\ ,
$$
and
$$
\bfig\scalefactor{.8}
\square/`@{>}|{\bb}`@{>}|{\bb}`=/[A'`C`C'`C';`v'`c_*`]

\square(0,500)/=`@{>}|{\bb}`@{>}|{\bb}`/[A`A`A'`C;`a_*`v`]

\place(250,500)[{\scriptstyle \beta}]

\square(500,0)/>``@{>}|{\bb}`>/[C`D`C'`D';``d_*`g']

\square(500,500)/>``@{>}|{\bb}`>/[A`B`C`D;f``w`g]

\place(750,250)[{\scriptstyle *}]

\place(750,750)[{\scriptstyle \alpha}]

\place(1300,500)[=]

\square(1600,0)/>`@{>}|{\bb}`@{>}|{\bb}`>/[A'`B'`C'`D';f'`v'`w'`g']

\square(1600,500)/>`@{>}|{\bb}`@{>}|{\bb}`/[A`B`A'`B';f`a_*`b_*`]

\place(1850,250)[{\scriptstyle \alpha'}]

\place(1850,750)[{\scriptstyle *}]

\square(2100,0)/``@{>}|{\bb}`=/[B'`D`D'`D';``d_*`]

\square(2100,500)/=``@{>}|{\bb}`/[B`B`B'`D;``w`]

\place(2350,500)[{\scriptstyle \gamma}]

\efig
$$
where the starred cells 
are the canonical ones gotten from the equations
$ g' c = d g $ and $ f' a = b f $ by ``sliding''.

\begin{proposition}
$ {\mathbb V}{\rm ar} ({\mathbb A}) $ is a strict double category.

\end{proposition}

\begin{proof}
We just have to check that cells compose horizontally and
vertically. We simply give a sketch of the proof.

Suppose we have two cells,
$$
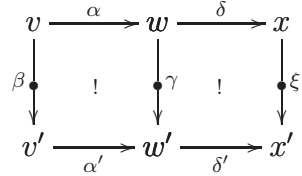
\bfig\scalefactor{.8}
\square/>`@{>}|{\bb}`@{>}|{\bb}`>/[v`w`v'`w';\alpha`\beta`\gamma`\alpha']

\place(250,250)[{\scriptstyle !}]

\square(500,0)/>``@{>}|{\bb}`>/[w`x`w'`x';\delta``\xi`\delta']

\place(750,250)[{\scriptstyle !}]

\efig
$$
i.e. we have

 \begin{center}
 \setlength{\unitlength}{.9mm}
 \begin{picture}(130,20)
 
\put(0,0){\framebox(20,20){}}
\put(30,0){\framebox(20,20){}}
\put(80,0){\framebox(20,20){}}
\put(110,0){\framebox(20,20){}}

\put(10,0){\line(0,1){20}}
\put(40,0){\line(0,1){20}}
\put(90,0){\line(0,1){20}}
\put(120,0){\line(0,1){20}}

\put(10,10){\line(1,0){10}}
\put(30,10){\line(1,0){10}}
\put(90,10){\line(1,0){10}}
\put(110,10){\line(1,0){10}}

\put(25,10){\makebox(0,0){$=$}}
\put(65,10){\makebox(0,0){and}}
\put(105,10){\makebox(0,0){$=$}}

\put(5,10){\makebox(0,0){$\scriptstyle \beta$}}
\put(15,5){\makebox(0,0){$\scriptstyle *$}}
\put(15,15){\makebox(0,0){$\scriptstyle \alpha$}}
\put(15,15){\makebox(0,0){$\scriptstyle \alpha$}}
\put(35,5){\makebox(0,0){$\scriptstyle \alpha'$}}
\put(35,15){\makebox(0,0){$\scriptstyle *$}}
\put(45,10){\makebox(0,0){$\scriptstyle \gamma$}}
\put(85,10){\makebox(0,0){$\scriptstyle \gamma$}}
\put(95,5){\makebox(0,0){$\scriptstyle *$}}
\put(95,15){\makebox(0,0){$\scriptstyle \delta$}}
\put(115,5){\makebox(0,0){$\scriptstyle \delta'$}}
\put(115,15){\makebox(0,0){$\scriptstyle *$}}
\put(125,10){\makebox(0,0){$\scriptstyle \xi$}}

\put(135,0){.}
 
 \end{picture}
 \end{center}

\noindent Thus

 \begin{center}
 \setlength{\unitlength}{.9mm}
 \begin{picture}(170,20)
 
\put(0,0){\framebox(20,20){}}
\put(30,0){\framebox(30,20){}}
\put(70,0){\framebox(30,20){}}
\put(110,0){\framebox(30,20){}}
\put(150,0){\framebox(20,20){}}
%
%
\put(10,0){\line(0,1){20}}
\put(40,0){\line(0,1){20}}
\put(50,0){\line(0,1){20}}
\put(80,0){\line(0,1){20}}
\put(90,0){\line(0,1){20}}
\put(120,0){\line(0,1){20}}
\put(130,0){\line(0,1){20}}
\put(160,0){\line(0,1){20}}
%
%
\put(10,10){\line(1,0){10}}
\put(40,10){\line(1,0){20}}
\put(70,10){\line(1,0){10}}
\put(90,10){\line(1,0){10}}
\put(110,10){\line(1,0){20}}
\put(150,10){\line(1,0){10}}

%
%
\put(25,10){\makebox(0,0){$=$}}
\put(65,10){\makebox(0,0){$=$}}
\put(105,10){\makebox(0,0){$=$}}
\put(145,10){\makebox(0,0){$=$}}

%
%
\put(5,10){\makebox(0,0){$\scriptstyle \beta$}}
\put(15,5){\makebox(0,0){$\scriptstyle *$}}
\put(15,15){\makebox(0,0){$\scriptstyle \delta \alpha$}}
\put(35,10){\makebox(0,0){$\scriptstyle \beta$}}
\put(45,5){\makebox(0,0){$\scriptstyle *$}}
\put(45,15){\makebox(0,0){$\scriptstyle \alpha$}}
\put(55,5){\makebox(0,0){$\scriptstyle *$}}
\put(55,15){\makebox(0,0){$\scriptstyle \delta$}}
\put(75,5){\makebox(0,0){$\scriptstyle \alpha'$}}
\put(75,15){\makebox(0,0){$\scriptstyle *$}}
\put(85,10){\makebox(0,0){$\scriptstyle \gamma$}}
\put(95,5){\makebox(0,0){$\scriptstyle *$}}
\put(95,15){\makebox(0,0){$\scriptstyle \delta$}}
\put(115,5){\makebox(0,0){$\scriptstyle \alpha'$}}
\put(115,15){\makebox(0,0){$\scriptstyle *$}}
\put(125,5){\makebox(0,0){$\scriptstyle \delta'$}}
\put(125,15){\makebox(0,0){$\scriptstyle *$}}
\put(135,10){\makebox(0,0){$\scriptstyle \xi$}}
\put(155,5){\makebox(0,0){$\scriptstyle \alpha' \delta'$}}
\put(155,15){\makebox(0,0){$\scriptstyle *$}}
\put(165,10){\makebox(0,0){$\scriptstyle \xi$}}
\put(175,0){.}
 
 \end{picture}
 \end{center}

\noindent Consider cells
$$
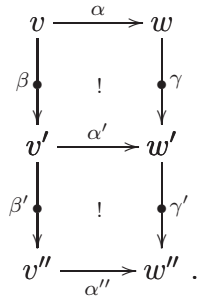
\bfig\scalefactor{.8}
\square/>`@{>}|{\bb}`@{>}|{\bb}`>/[v'`w'`v''`w''\rlap{\ .};\alpha'`\beta'`\gamma'`\alpha'']

\place(250,250)[{\scriptstyle !}]

\square(0,500)/>`@{>}|{\bb}`@{>}|{\bb}`/[v`w`v'`w';\alpha`\beta`\gamma`]

\place(250,750)[{\scriptstyle !}]

\efig
$$
We did not say, but vertical composition of arrows in
$ {\mathbb V}{\rm ar} ({\mathbb A}) $ is given by horizontal
composition of retrocells. It could not be otherwise given their
boundaries. Then we have the following

 \begin{center}
 \setlength{\unitlength}{.9mm}
 \begin{picture}(170,30)
 
\put(0,5){\framebox(20,20){}}
\put(30,0){\framebox(30,30){}}
\put(70,0){\framebox(30,30){}}
\put(110,0){\framebox(30,30){}}
\put(150,5){\framebox(20,20){}}
\put(10,5){\line(0,1){20}}
\put(40,0){\line(0,1){30}}
\put(50,0){\line(0,1){30}}
\put(80,0){\line(0,1){30}}
\put(90,0){\line(0,1){30}}
\put(120,0){\line(0,1){30}}
\put(130,0){\line(0,1){30}}
\put(160,5){\line(0,1){20}}
\put(10,15){\line(1,0){10}}
\put(30,20){\line(1,0){10}}
\put(40,10){\line(1,0){20}}
\put(50,20){\line(1,0){10}}
\put(70,20){\line(1,0){20}}
\put(80,10){\line(1,0){20}}
\put(110,10){\line(1,0){10}}
\put(110,20){\line(1,0){20}}
\put(130,10){\line(1,0){10}}
\put(150,15){\line(1,0){10}}
%
\put(25,15){\makebox(0,0){$=$}}
\put(65,15){\makebox(0,0){$=$}}
\put(105,15){\makebox(0,0){$=$}}
\put(145,15){\makebox(0,0){$=$}}
%
%
\put(5,15){\makebox(0,0){$\scriptstyle \beta' \bdot \beta$}}
\put(15,10){\makebox(0,0){$\scriptstyle *$}}
\put(15,20){\makebox(0,0){$\scriptstyle \alpha$}}
\put(35,10){\makebox(0,0){$\scriptstyle \beta'$}}
\put(35,25){\makebox(0,0){$\scriptstyle =$}}
\put(45,5){\makebox(0,0){$\scriptstyle =$}}
\put(45,20){\makebox(0,0){$\scriptstyle \beta$}}
\put(55,5){\makebox(0,0){$\scriptstyle *$}}
\put(55,25){\makebox(0,0){$\scriptstyle \alpha$}}
\put(75,10){\makebox(0,0){$\scriptstyle \beta'$}}
\put(75,25){\makebox(0,0){$\scriptstyle =$}}
\put(85,5){\makebox(0,0){$\scriptstyle *$}}
\put(85,15){\makebox(0,0){$\scriptstyle \alpha'$}}
\put(85,25){\makebox(0,0){$\scriptstyle *$}}
\put(95,5){\makebox(0,0){$\scriptstyle =$}}
\put(95,20){\makebox(0,0){$\scriptstyle \gamma$}}
\put(115,5){\makebox(0,0){$\scriptstyle \alpha''$}}
\put(115,15){\makebox(0,0){$\scriptstyle *$}}
\put(115,25){\makebox(0,0){$\scriptstyle *$}}
\put(125,10){\makebox(0,0){$\scriptstyle \gamma'$}}
\put(125,25){\makebox(0,0){$\scriptstyle =$}}
\put(135,5){\makebox(0,0){$\scriptstyle =$}}
\put(135,20){\makebox(0,0){$\scriptstyle \gamma$}}
\put(155,10){\makebox(0,0){$\scriptstyle \alpha''$}}
\put(155,20){\makebox(0,0){$\scriptstyle *$}}
\put(165,15){\makebox(0,0){$\scriptstyle \gamma' \bdot \gamma$}}

\put(175,5){.}
 \end{picture}
 \end{center}

\noindent So horizontal and vertical composition of cells are again cells.

Identities pose no problem.

\end{proof}

\begin{proposition}
A cell
$$
\bfig\scalefactor{.8}
\square/>`@{>}|{\bb}`@{>}|{\bb}`>/[A`B`C`D;f`v`w`g]

\place(250,250)[{\scriptstyle \alpha}]

\efig
$$
in $ {\mathbb A} $ is a commuter cell iff $ \alpha \colon v \to w $
has a companion in $ {\mathbb V}{\rm ar} {\mathbb A} $.

\end{proposition}

\begin{proof}
A companion $ \beta \colon v \tod w $ for $ \alpha $ will
have cells
$$
\bfig\scalefactor{.8}
\square/>`@{>}|{\bb}`@{>}|{\bb}`=/[v`w`w`w;\alpha`\beta`\id_w`]

\place(250,250)[{\scriptstyle !}]

\place(850,250)[\mbox{and}]

\square(1200,0)/=`@{>}|{\bb}`@{>}|{\bb}`>/[v`v`v`w;`\id_v`\beta`\alpha]

\place(1450,250)[{\scriptstyle !}]

\efig
$$
i.e. it is a retrocell
$$
\bfig\scalefactor{.8}
\square/>`@{>}|{\bb}`@{>}|{\bb}`>/[A`B`C`D;f'`v`w`g']

\morphism(340,250)/=>/<-200,0>[`;\beta]

\efig
$$
making the following cubes ``commute''

$$
\bfig\scalefactor{.9}
\node a(300,0)[D]
\node b(800,0)[D]
\node c(0,250)[B]
\node d(300,500)[C]
\node e(800,500)[D]
\node f(0,800)[A]
\node g(500,800)[B]

\arrow|b|/=/[a`b;]
\arrow|l|/@{>}|{\bb}/[c`a;w]
\arrow|r|/>/[d`a;g']
\arrow|r|/=/[e`b;]
\arrow|b|/>/[d`e;g]
\arrow|l|/>/[f`c;f']
\arrow|l|/@{>}|{\bb}/[f`d;v]
\arrow|r|/@{>}|{\bb}/[g`e;w]
\arrow|a|/>/[f`g;f]

\place(430,650)[{\scriptstyle \alpha}]

\morphism(150,300)|l|/=>/<0,200>[`;\beta]

\node f'(1500,800)[A]
\node g'(2000,800)[B]
\node c'(1500,300)[B]
\node d'(2000,300)[B]
\node a'(1800,0)[D]
\node b'(2300,0)[D]
\node e'(2300,400)[D]

\arrow|a|/>/[f'`g';f]
\arrow|l|/>/[f'`c';f']
\arrow|l|/=/[g'`d';]
\arrow|r|/@{>}|{\bb}/[g'`e';w]
\arrow|a|/=/[c'`d';]
\arrow|l|/@{>}|{\bb}/[c'`a';w]
\arrow|b|/=/[a'`b';]
\arrow/@{>}|{\bb}/[d'`b';w]
\arrow|r|/=/[e'`b';]

\place(1950,150)[{\scriptstyle 1_w}]

\morphism(2130,320)|r|/=>/<0,200>[`;]
\place(2215,335)[{\scriptstyle 1_w}]

\place(1150,500)[\mbox{``=''}]

\efig
$$
and
$$
\bfig\scalefactor{.9}
\node a(300,0)[C]
\node b(800,0)[D]
\node c(0,250)[A]
\node d(300,500)[C]
\node e(800,500)[C]
\node f(0,800)[A]
\node g(500,800)[A]

\arrow|b|/>/[a`b;g]
\arrow|l|/@{>}|{\bb}/[c`a;v]
\arrow|r|/=/[d`a;]
\arrow|r|/>/[e`b;g']
\arrow|b|/=/[d`e;]
\arrow|l|/=/[f`c;]
\arrow|l|/@{>}|{\bb}/[f`d;v]
\arrow|r|/@{>}|{\bb}/[g`e;v]
\arrow|a|/=/[f`g;]

\place(430,650)[{\scriptstyle 1_v}]

\morphism(160,320)|l|/=>/<0,200>[`;1_v]

\node f'(1500,800)[A]
\node g'(2000,800)[A]
\node c'(1500,300)[A]
\node d'(2000,300)[B]
\node a'(1800,0)[C]
\node b'(2300,0)[D]
\node e'(2300,400)[C]

\arrow|a|/=/[f'`g';]
\arrow|l|/=/[f'`c';]
\arrow|l|/>/[g'`d';f']
\arrow|r|/@{>}|{\bb}/[g'`e';v]
\arrow|a|/>/[c'`d';f]
\arrow|l|/@{>}|{\bb}/[c'`a';v]
\arrow|b|/>/[a'`b';g]
\arrow/@{>}|{\bb}/[d'`b';w]
\arrow|r|/>/[e'`b';g']

\place(1950,150)[{\scriptstyle \alpha}]

\morphism(2150,320)|l|/=>/<0,200>[`;\beta]

\place(1150,500)[\mbox{``=''}]

\efig
$$

\noindent So, first of all $ f = f' $ and $ g = g' $. The first
``equation'' says
$$
\bfig\scalefactor{.8}
\square/`@{>}|{\bb}`@{>}|{\bb}`=/[B`C`D`D;`w`g_*`]
\square(0,500)/=`@{>}|{\bb}`@{>}|{\bb}`/[A`A`B`C;`f_*`v`]

\place(250,500)[{\scriptstyle \beta}]

\square(500,0)/>``=`=/[C`D`D`D;g```]

\place(750,250)[{ \lrcorner}]

\square(500,500)/>``@{>}|{\bb}`/[A`B`C`D;f``w`]

\place(750,750)[{\scriptstyle \alpha}]

\place(1350,500)[=]

\square(1700,0)/=`@{>}|{\bb}`@{>}|{\bb}`=/[B`B`D`D;`w`w`]

\place(1950,250)[{\scriptstyle 1_w}]

\square(1700,500)/>`@{>}|{\bb}`=`/[A`B`B`B;f`f_*``]

\place(1950,750)[{\scriptstyle \lrcorner}]

\square(2200,0)/=```=/<500,1000>[B`B`D`D;```]

\morphism(2700,1000)|r|/@{>}|{\bb}/<0,-500>[B`D;w]

\morphism(2700,500)/=/<0,-500>[D`D;]

\place(2450,500)[{\scriptstyle \cong}]

\efig
$$
which by sliding is equivalent to $ \widehat{\alpha} \beta = 1_{w \bdot f_*} $.
Similarly the second equation says $ \beta \widehat{\alpha} = 1_{g_* \bdot v} $.

\end{proof}

We end by acknowledging the ``triple category in the room''. The
cubes we have been discussing are clearly the triple cells of a triple
category $ {\mathfrak{Ret}} {\mathbb A} $. We orient the cubes to
be in line with our intercategories conventions of \cite{GraPar15}
where the faces of the cubes are horizontal, vertical (left and right),
and basic (front and back) in decreasing order of strictness (or
fancyness). The order here will be commutative, cell, and retrocell.

\begin{itemize}

	\item[1.] Objects are the objects of $ {\mathbb A} $, ($ A, A', B, ..$)
	
	\item[2.] Transversal arrows are the horizontal arrows of $ {\mathbb A} $,
	($ f, f', g, g' $)
	
	\item[3.] Horizontal arrows are the horizontal arrows of $ {\mathbb A} $,
	($ a, b, c, d $)
	
	\item[4.] Vertical arrows are the vertical arrows of $ {\mathbb A} $,
	($ v, v', w, w' $)
	
	\item[5.] Horizontal cells are commutative squares of horizontal arrows
	
	\item[6.] Vertical cells are double cells in $ {\mathbb A} $, ($ \alpha, \alpha' $)
	
	\item[7.] Basic cells are retrocells in $ {\mathbb A} $, ($ \beta, \gamma $)
	
	\item[8.] Triple cells are ``commutative'' cubes as discussed above

\end{itemize}

$$
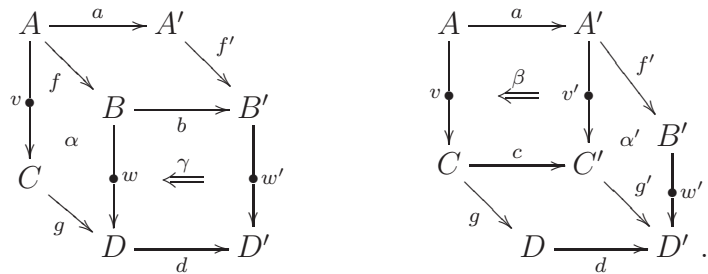
\bfig\scalefactor{.9}
\node a(300,0)[D]
\node b(800,0)[D']
\node c(0,250)[C]
\node d(300,500)[B]
\node e(800,500)[B']
\node f(0,800)[A]
\node g(500,800)[A']

\arrow|b|/>/[a`b;d]
\arrow|l|/>/[c`a;g]
\arrow|r|/@{>}|{\bb}/[d`a;w]
\arrow|r|/@{>}|{\bb}/[e`b;w']
\arrow|b|/>/[d`e;b]
\arrow|l|/@{>}|{\bb}/[f`c;v]
\arrow|l|/>/[f`d;f]
\arrow|r|/>/[g`e;f']
\arrow|a|/>/[f`g;a]

\place(150,400)[{\scriptstyle \alpha}]

\morphism(650,250)|l|/=>/<-200,0>[`;\gamma]

\node f'(1500,800)[A]
\node g'(2000,800)[A']
\node c'(1500,300)[C]
\node d'(2000,300)[C']
\node a'(1800,0)[D]
\node b'(2300,0)[D'\rlap{\ .}]
\node e'(2300,400)[B']

\arrow|a|/>/[f'`g';a]
\arrow|l|/@{>}|{\bb}/[f'`c';v]
\arrow|l|/@{>}|{\bb}/[g'`d';v']
\arrow|r|/>/[g'`e';f']
\arrow|a|/>/[c'`d';c]
\arrow|l|/>/[c'`a';g]
\arrow|b|/>/[a'`b';d]
\arrow/>/[d'`b';g']
\arrow|r|/@{>}|{\bb}/[e'`b';w']

\morphism(1850,550)/=>/<-200,0>[`;\beta]

\place(2150,400)[{\scriptstyle \alpha'}]

\efig
$$

We leave the details to the interested reader.

\bibliography{Pare-references,TAC-references}

\begin{thebibliography}{10}

\bibitem{Agu97}
Marcelo Aguiar.
\newblock {\em Internal categories and quantum groups}.
\newblock PhD thesis, Cornell University, 1997.

\bibitem{Cla20}
Bryce Clarke.
\newblock Internal split opfibrations and cofunctors.
\newblock {\em Theory and Applications of Categories}, 35(44):1608--1633, 2020.

\bibitem{Cla22}
Bryce Clarke.
\newblock {\em The double category of lenses}.
\newblock PhD thesis, Macquarie University, 2022.

\bibitem{ClaDim22}
Bryce Clarke and Matthew~Di Meglio.
\newblock An introduction to enriched cofunctors.
\newblock Preprint, 2022.
\newblock arXiv:2209.01144.

\bibitem{DawPar93B}
Robert Dawson and Robert Par\'e.
\newblock General associativity and general composition for double categories.
\newblock {\em Cahiers de Topologie et G{\'e}om{\'e}trie Diff{\'e}rentielle
  Cat{\'e}goriques}, 34(1):57--79, 1993.

\bibitem{Day74}
Brian Day.
\newblock Limit spaces and closed span categories.
\newblock In G.~M. Kelly, editor, {\em Category Seminar}, volume 420 of {\em
  Lecture Notes in Mathematics}, pages 65--74, 1974.

\bibitem{FioGamKoc12}
Thomas Fiore, Nicola Gambino, and Joachim Kock.
\newblock Double adjunctions and free monads.
\newblock {\em Cahiers de Topologie et G{\'e}om{\'e}trie Diff{\'e}rentielle
  Cat{\'e}goriques}, 53(4):242--306, 2012.

\bibitem{FioGamKoc11}
Thomas~M. Fiore, Nicola Gambino, and Joachim Kock.
\newblock Monads in double categories.
\newblock {\em Journal of Pure and Applied Algebra}, 215(6):1174--1197, 2011.

\bibitem{Gra20}
Marco Grandis.
\newblock {\em Higher Dimensional Categories: From Double to Multiple
  Categories}.
\newblock World Scientific, 2019.

\bibitem{GraPar04}
Marco Grandis and Robert Par\'e.
\newblock Adjoint for double categories.
\newblock {\em Cahiers de Topologie et G{\'e}om{\'e}trie Diff{\'e}rentielle
  Cat{\'e}goriques}, 45(3):193--240, 2004.

\bibitem{GraPar08}
Marco Grandis and Robert Par\'e.
\newblock Kan extensions in double categories {(On weak double categories, Part
  III)}.
\newblock {\em Theory and Applications of Categories}, 20(8):152--185, 2008.

\bibitem{GraPar15}
Marco Grandis and Robert Par\'e.
\newblock Intercategories.
\newblock {\em Theory and Applications of Categories}, 30(38):1215--1255, 2015.

\bibitem{Kou14}
Seerp~Roald Koudenburg.
\newblock On pointwise {Kan} extensions in double categories.
\newblock {\em Theory and Applications of Categories}, 29(27):781--818, 2014.

\bibitem{Lam66}
Joachim Lambek.
\newblock {\em Lectures on Rings and Modules}.
\newblock Blaisdell Publishing, 1966.

\bibitem{Lei04}
Tom Leinster.
\newblock {\em Higher Operads, Higher Categories}, volume 298 of {\em London
  Mathematical Society Lecture Note Series}.
\newblock Cambridge University Press, 2004.

\bibitem{DiM22}
Matthew~Di Meglio.
\newblock The category of asymmetric lenses and its proxy pullbacks.
\newblock Master's thesis, Macquarie University, 2022.

\bibitem{Par21}
Robert Par\'e.
\newblock Morphisms of rings.
\newblock In Claudia Casadio and Philip~J. Scott, editors, {\em Joachim Lambek:
  The Interplay of Mathematics, Logic, and Linguistics}, volume~20 of {\em
  Outstanding Contributions to Logic}, pages 271--298. Springer, 2021.

\bibitem{Shu08}
Michael Shulman.
\newblock Framed bicategories and monoidal fibrations.
\newblock {\em Theory and Applications of Categories}, 20(18):650--738, 2008.

\bibitem{Str72}
Ross Street.
\newblock The formal theory of monads.
\newblock {\em Journal of Pure and Applied Algebra}, 2(2):149--168, 1972.

\end{thebibliography}


\end{document}